\documentclass{article}

\usepackage[T1]{fontenc}
\usepackage[utf8]{inputenc}

\usepackage[margin=1.3in]{geometry}
\usepackage{graphicx}

\usepackage{mathtools}%
\usepackage{amsthm}
\usepackage{amsxtra}%
\usepackage{amsfonts}%
\usepackage{amssymb}%
\usepackage{authblk}

\usepackage[
natbib,
maxcitenames=3, 
maxbibnames=10,  
sorting=ydnt,
url=false,
doi=false,
isbn=false,
sortcites,
defernumbers,
backref,
backend=biber
]{biblatex}
\usepackage{todonotes}
\usepackage{url}

\usepackage{xfrac}
\usepackage{nicefrac}

\usepackage{bbm}
\usepackage[shortlabels]{enumitem}

\usepackage{hyperref}
\usepackage[nameinlink, capitalise, noabbrev]{cleveref}
\mathtoolsset{showonlyrefs}
\makeatletter
\renewcommand*\MT@newlabel[1]{%
	\global\@namedef{MT_r_#1}{}%
	\global\@namedef{MT_r_{#1}}{}%
}
\makeatother

\newtheorem{theorem}{Theorem}
\newtheorem{lemma}{Lemma}[section]
\newtheorem{corollary}[lemma]{Corollary}
\newtheorem{example}[lemma]{Example}
\newtheorem{proposition}[lemma]{Proposition}
\theoremstyle{definition}
\newtheorem{remark}[lemma]{Remark}
\newtheorem{definition}[lemma]{Definition}

\numberwithin{equation}{section}

\newcommand{\norm}[1]{\left\Vert #1 \right\Vert}

\newcommand{\abs}[1]{\left\vert #1 \right\vert}
\DeclareMathOperator{\divtmp}{div}
\renewcommand{\div}{\divtmp}
\DeclareMathOperator{\argmin}{arg\,min}
\DeclareMathOperator{\argmax}{arg\,max}
\newcommand{\st}{\,:\,}
\DeclareMathOperator{\supp}{supp}
\newcommand{\dx}{\,\mathrm{d}x}
\renewcommand{\d}{\,\mathrm{d}}

\newcommand{\eps}{\varepsilon}

\newcommand{\restr}{\mathbin{\vrule height 1.6ex depth 0pt width
0.13ex\vrule height 0.13ex depth 0pt width 1.3ex}}

\DeclareMathOperator{\grad}{grad}
\DeclareMathOperator*{\Esssup}{ess\ sup}							
\DeclareMathOperator*{\Essinf}{ess\ inf}							
	
\newcommand{\esssup}[2]{\Esssup_{#2}{#1}}						
\newcommand{\nesssup}[2]{\nu\mhyphen\Esssup_{#2}{#1}}
\newcommand{\nessinf}[2]{\nu\mhyphen\Essinf_{#2}{#1}}

\newcommand{\IRn}{{\IR^{\kern-.7pt n}}}							
\newcommand{\IRN}{{\IR^{\kern-.7pt N}}}
\newcommand{\IRm}{{\IR^{\kern-.7pt n}}}	
\newcommand{\IRM}{{\IR^{\kern-.7pt M}}}

\newcommand{\set}[2]{\left\{#1 \,\colon\, #2\right\}}			
\newcommand{\dotproduct}[2]{\left\langle #1,#2\right\rangle}	

\newcommand{\ball}[2]{\mathrm{B}_{#1}(#2)}						
\newcommand{\clball}[2]{\overline{\mathrm{B}}_{#1}(#2)}						

\newcommand{\Gr}{\mathrm{Gr}\,}                                 
\newcommand{\dd}{\,\mathrm{d}}	
\newcommand{\filll}{\ \cdot\ }									

\newcommand{\tv}{\operatorname{TV}}

\newcommand{\corr}{\twoheadrightarrow}

\newcommand{\LL}[1]{\mathsf{\large{L}}^{\kern-1.5pt#1}}			
\newcommand{\Cc}[1]{{C}^{#1}_c}	
\newcommand{\Cci}{\Cc{\infty}}						

\newcommand{\p}{\partial}						
\newcommand{\Ome}{\Omega}
\renewcommand{\phi}{\varphi}
\newcommand{\vrho}{\varrho}
\renewcommand{\rho}{\vrho}

\newcommand{\Frho}{{\boldsymbol{\vrho}}}
\newcommand{\FPsi}{\boldsymbol{\Psi}}

\newcommand{\serifm}{\mathsf{m}}
\newcommand{\Boldserifm}{\bm{\serifm}}
\newcommand{\serifPsi}{\mathsf{\Psi}}
\newcommand{\BoldserifPsi}{\bm{\serifPsi}}
\newcommand{\Deps}{\mathsf{D}_\eps}
\newcommand{\clDeps}{\overline{\mathsf{D}}_\eps}

\mathchardef\mhyphen="2D

\definecolor{bordercolor}{RGB}{0, 74, 143}

\usepackage[bb=dsserif]{mathalpha}
\usepackage{bm}
\newcommand\Ie{\ifdefined\mathbbb\mathbbb{1}%
	\else\boldsymbol{\mathbb{1}}%
	\fi}														

\def\Xint#1{\mathchoice
	{\XXint\displaystyle\textstyle{#1}}%
	{\XXint\textstyle\scriptstyle{#1}}%
	{\XXint\scriptstyle\scriptscriptstyle{#1}}%
	{\XXint\scriptscriptstyle\scriptscriptstyle{#1}}%
	\!\int}
\def\XXint#1#2#3{{\setbox0=\hbox{$#1{#2#3}{\int}$}
		\vcenter{\hbox{$#2#3$}}\kern-.5\wd0}}

\def\dashint{\Xint-}

\newcommand{\mint}{\dashint}

\DeclareMathOperator*{\wlimop}{w\mhyphen lim}
\newcommand{\wlim}[1]{\wlimop_{#1}}

\newcommand{\IE}{\mathbbm{E}}

\newcommand{\IN}{\mathbbm{N}}

\newcommand{\IR}{\mathbbm{R}}

\newcommand{\A}{\mathcal{A}}

\renewcommand{\H}{\mathcal{H}}

\newcommand{\M}{\mathcal{M}}

\renewcommand{\P}{\mathcal{P}}	

\newcommand{\R}{\mathbb{R}}

\newcommand{\X}{\mathcal{X}}
\newcommand{\Y}{\mathcal{Y}}

\newcommand{\BB}{\mathfrak{B}}

\newcommand{\MM}{\mathfrak{M}}

\newcommand{\PP}{\mathfrak{P}}

\begin{document}

\title{Duality for the Adversarial Total Variation}
\author{Leon Bungert}
\author{Lucas Schmitt} 
\affil{Institute of Mathematics, Center for Artificial Intelligence and Data Science (CAIDAS), University of Würzburg, Emil-Fischer-Str. 40, 97074 Würzburg, Germany, Email: \texttt{\{leon.bungert, lucas.schmitt\}@uni-wuerzburg.de}}

\maketitle

\begin{abstract}
    Adversarial training of binary classifiers can be reformulated as regularized risk minimization involving a nonlocal total variation.
    Building on this perspective, we establish a characterization of the subdifferential of this total variation using duality techniques. 
    To achieve this, we derive a dual representation of the nonlocal total variation and a related integration-by-parts formula, involving a nonlocal gradient and divergence. 
    We provide such duality statements both in the space of continuous functions vanishing at infinity on proper metric spaces and for the space of essentially bounded functions on separable metric spaces.
    Furthermore, under some additional conditions we provide characterizations of the subdifferential in these settings. 
\end{abstract}

\section{Introduction}

In recent years, machine learning algorithms and, in particular, deep neural networks have enjoyed unprecedented success; however, since more than a decade ago it is known that they are also susceptible to \emph{adversarial attacks} \cite{szegedy2014}, which are small carefully chosen perturbations of the input that lead to severe misclassifications. 
While this was initially a major concern for image classification applications in safety-critical domains, more recently, in \cite{zou2023universal} it was shown that adversarial attacks can also be crafted for LLMs. 
One prominent approach to improve robustness is \emph{adversarial training}, proposed in \cite{madry2018towards} building on earlier ideas from \cite{goodfellow2015explaining}.
The key idea is to replace the standard training risk by the risk evaluated on adversarial attacks, thereby simulating the presence of an adversary already during training. 
Mathematically speaking, the resulting method is a robust optimization problem of the form
\begin{align}\label{eq:Adversarial Training}
	\inf_{u\in\H}\IE_{(x,y)\sim\mu}\left[\sup_{\tilde{x}\in\ball{\eps}{x}}\ell (u(\tilde{x}),y)\right],
\end{align}
where $\mu\in \P(\X\times\Y)$ is a probability measure defined on the cartesian product of the feature space $\X$ and the label space $\Y$, modeling the distribution of training data. 
The hypothesis class $\H$ may be any family of measurable functions from $\X$ to $\Y$ appropriate for the specific task. 
The value of $\eps\geq 0$ is called the \emph{adversarial budget} and indicates how ``strong'' the adversarial attacks are allowed to be. Depending on the specific machine learning task that is considered, one chooses an appropriate \emph{loss function} $\ell\colon \Y\times\Y\to\IR$. 

With these ingredients at hand, the rationale behind \eqref{eq:Adversarial Training} is to enforce that for any training example $(x,y)\sim\mu$, an optimal hypothesis satisfies $u(\tilde x)\approx y$ for all $\tilde x\in\ball{\eps}{x}$ and is thereby robust for perturbations of size at most $\eps$.

In contrast, the mathematical understanding of \eqref{eq:Adversarial Training} is highly non-trivial, as we will explain in the following:
Proving existence is difficult due to the presence of the supremum in the objective function.
For non-parametric binary classification, where $\Y=\{0,1\}$, $\H$ is a class of measurable functions and $\ell(y_1,y_2)=\abs{y_1-y_2}$ is the $0$-$1$-loss, the existence of solutions is shown in \cite{awasthi2021existence} and \cite{bungert2023geometry}. The generalization to multiclass classification is treated in \cite{trillos2024existence} and the case of general loss functions was treated in \cite{siethoff2025existence}.
Uniqueness of solutions can not be expected in general but regular solutions can be constructed for positive $\eps$, see \cite{bungert2023geometry}.
The asympotics of \eqref{eq:Adversarial Training} for the adversarial budget $\eps$ tending to zero were studied using Gamma-convergence in \cite{bungert2024gamma,lunk2026explaining} and geometric techiques in \cite{morris2026tradeoff,morris2026uniform}, relations to mean curvature flow were established in \cite{bungert2024mean,trillos2022adversarial}, optimal transport perspectives on adversarial training can be found in \cite{trillos2024optimal,trillos2023multimarginal}, and an overview of recent developments on this topic is given in \cite{trillos2023analytical}.
With respect to adversarial attacks, i.e., maximizers of $\sup_{\tilde x\in\ball{\eps}{x}}\ell(u(\tilde x),y)$ we would like to point to \cite{weigand2026adversarial} where the authors connect the projected gradient ascent methods typically used for approximating adversarial attacks numerically to metric gradient flows. 
Note that also the algorithmic minimization of \eqref{eq:Adversarial Training} is challenging since the presence of the supremum destroys differentiability of the objective.
This fact is largely ignored by the machine learning community since the advent of automatic differentiation allows one to apply gradient-based optimizers to basically any function.
By the classical Danskin theorem, however, this only makes sense if the supremum in \eqref{eq:Adversarial Training} is attained at a unique point in the closure of $\ball{\eps}{x}$; an assumption that is too strong for most real life situations.

In this paper we make a first step towards a rigorous treatment of this optimization problem by using duality and techniques from non-smooth convex analysis. 
We build on the key insight from \cite{bungert2023geometry} which, in the binary classification setting, connects adversarial training to a regularization problem of the form
\begin{align}
    \inf_{u\in\H}\IE_{(x,y)\sim\mu}\big[\ell(u(x),y)\big]+\lambda \mathcal R(u),
\end{align}
where $\mathcal R$ is the regularization functional and $\lambda>0$ a parameter that controls the strength of regularization. 
In classical linear parametric learning settings the connection between adversarial training and regularization has been explored, for example, in \cite{blanchet2019robust,blanchet2021statistical}.
The regularization interpretation identified in \cite{bungert2023geometry} features a nonlocal notion of length of the decision boundary.
Their precise result states that in the non-parametric binary setting minimizers of \eqref{eq:Adversarial Training} exist and can be constructed from those of the convex optimization problem
\begin{align}\label{eq:reg problem}
	\inf_{\substack{u\in L^\infty(\X;\nu)\\0\leq u\leq 1,\,\nu\mhyphen\text{a.e.}}} \IE_{(x,y)\sim\mu} \big[\vert u(x)-y\vert\big]+\eps\,\nu\mhyphen\tv_\eps(u)
\end{align}
where $\nu$ is a reference measure that dominates the class-conditional data distributions and is supported on an $\eps$-neighborhood around their support. 
The nonlocal \emph{adversarial total variation} is given by
\begin{align}\label{eq:TV functional}
	\nu\mhyphen\tv_\eps(u)\coloneq\int_\X \frac{\nesssup{u}{\ball{\eps}{x}}-u(x)}{\eps}\dd\vrho_0(x)+\int_\X \frac{u(x)-\nessinf{u}{\ball{\eps}{x}}}{\eps}\dd\vrho_1(x),
\end{align}
where $\vrho_i \coloneq \mu(\filll\times\{i\})$ for $i\in\{0,1\}$ are the (un-normalized) class-conditional distributions of the points with label $i$.
Until now, several properties of this functional has been examined in \cite{bungert2023geometry,bungert2024gamma,bungert2024mean,lunk2026explaining}, see also \cite{trillos2024existence,siethoff2025existence} for the multi-class case.

In particular, in \cite{bungert2024gamma} it was shown that if $\X=\Omega\subset\IRN$ is an open subset of $\IRN$ the nonlocal total variation Gamma-converges to an anisotropic local total variation of the form
\begin{align}
     u\mapsto\int_\Omega \beta\left(x,\frac{Du}{\abs{Du}}(x)\right)\d\abs{Du}(x),\qquad u \in BV(\Omega),
\end{align}
where the weight function $\beta\colon\Omega\times\mathbb S^{N-1}\to\R$ reduces to $\beta(x,\nu)=\vrho_0(x)+\vrho_1(x)$ if the class-conditional distributions $\vrho_0,\vrho_1$ possess a continuous density with respect to the Lebesgue measure.
Here $BV(\Omega)$ is the space of functions of bounded variation on $\Omega$, referring to all functions $u\in L^1(\Omega)$ such that the distributional gradient of $u$ is a finite vector-valued Radon measure $Du$, see the monograph \cite{ambrosio2000functionsofBV} for precise definitions and properties.
Among other properties, like the coarea formula and positive homogeneity of $\nu\mhyphen\tv_\eps$, this Gamma-convergence justifies the nomenclature ``total variation''. 
Note that a local total variation of the form $\int_\Omega \rho(x)\d\abs{Du}(x)$ for sufficiently regular $\rho$ can be dualized as follows \cite{ambrosio2000functionsofBV}:
\begin{align}
    \int_\Omega \rho(x)\d\abs{Du}(x)
    =
    \sup\left\lbrace
    -\int_\Omega u(x)\div\left(\rho(x)\phi(x)\right)\d x
    \st 
    \phi \in C_c^\infty(\Omega;\IRN),\,\abs{\phi(x)}\leq 1\;\forall x\in\Omega
    \right\rbrace.
\end{align}
In this and other similar dualization formulas we refer to the set of functions over which the supremum is taken as the set of test functions, which in this case is the set of compactly supported smooth vector fields on $\IRN$.
  
Using standard statements from convex analysis (cf. \cite{bredies2016pointwise,bungert2021nonlinear}) one can use this dualization to show that the subdifferential of the functional at some $u\in L^p(\Omega)\cap BV(\Omega)$ is given by all functions $u^*$ in the closure of the set
\begin{align}
    \left\lbrace 
    -\div\left(\rho\phi\right)
    \st 
    \phi \in C_c^\infty(\Omega;\IRN),\,\abs{\phi(x)}\leq 1\;\forall x\in\Omega
    \right\rbrace 
\end{align} 
in $L^q(\Omega)$ with $\frac{1}{p}+\frac{1}{q}=1$ which satisfy
\begin{align}
    \int_\Omega u^*(x)u(x)\d x = \int_\Omega \rho(x)\d\abs{Du}(x).
\end{align}
Using the theory of Anzellotti pairings \cite{anzellotti1983pairings}, it can further be shown that any subgradient can be written as $u^*=-\div\left(\rho\phi\right)$ where $\phi$ is an essentially bounded vector field that, in a suitable generalized sense, possesses a divergence and is parallel to the distributional gradient $Du$ of $u$, see \cite{bredies2016pointwise} for all details.
This subdifferential characterization has many applications, most notably, it is at the heart of the solution concept of the total variation flow, see \cite{andreu2001minimizing,andreu2002some}.

Another important application of the dualization formula for the total variation lies in numerical optimization. 
For instance, the seminal Rudin--Osher--Fatemi (ROF) model consists of the minimization of 
\begin{align}
    L^2(\Omega)\ni u \mapsto 
    \frac12\int_\Omega\abs{u-f}^2\d x
    +
    \lambda\int_\Omega\d\abs{Du},
\end{align}
where $f\in L^2(\Omega)$ is a given noisy image that needs to be denoised and $\lambda>0$ is a regularization parameter.
Here, we set $\int_\Omega\d\abs{Du}=\infty$ for $u\in L^2(\Omega)\setminus BV(\Omega)$.
The optimality conditions for this optimization problem are given by $u-f+\lambda u^* = 0$ where $u^*\in L^2(\Omega)$ is a subgradient of the total variation at $u$, i.e., $p=-\div\phi$ for a suitable vector field $\phi$.
This optimality condition, together with the subdifferential characterization, was used in \cite{chambolle2004algorithm} to derive a seminal algorithm to approximate minimizers of the image denoising problem above and related similar problems in a provably convergent way.
Later, the famous \emph{primal-dual algorithm} was devised in \cite{chambolle2011primaldual} which applies the dualization more directly to derive a non-smooth optimization algorithm. 
For example, we can reformulate the ROF model as the saddle-point problem
\begin{align}
    \inf_{u\in L^2(\Omega)}
    \sup_{\substack{\phi\in C_c^\infty(\Omega;\IRN)\\\abs{\phi(x)}\leq 1\;\forall x\in\Omega}}
    \frac{1}{2}\int_\Omega\abs{u-f}^2\d x
    -
    \lambda
    \int_\Omega u \div\phi \d x.
\end{align}
Note that---once discretized to a finite dimensional problem---the minimization problem in $u$ is smooth and convex whereas the maximization problem in $\phi$ is smooth and concave with a convex and closed constraint set of test functions. 
Hence, one can apply projected gradient methods to both variables and obtain the primal-dual algorithm from \cite{chambolle2011primaldual}.

Related to this, for the special case of $\X=\IRN$ and $\vrho_0=\vrho_1=\lambda^N$, \cite{cagnetti2026distributional} analyzed the optimality conditions of a nonlocal ROF-functional, where the adversarial total variation is used instead of the classical one.
This is used to construct a Minkowski-type mean curvature flow, employing the distributional approach introduced in \cite{chambolle2017existence}. 
The authors also derive a nonlocal divergence and calibrations that are similar to the ones in our work (we call those calibrations random walks).
In contrast to \cite{bungert2024mean}, where the authors also studied a nonlocal ROF-functional with the nonlocality parameter being equal to the time step size for the minimizing movement approximation of mean curvature flow, in \cite{cagnetti2026distributional} these parameters are decoupled which makes the resulting mean curvature flow nonlocal.

It should be noted that similar dualization ideas as for the ROF-problem can also be applied to the adversarial training problem in \eqref{eq:reg problem} (where the non-smoothness of the loss function does not pose a problem and can be ignored for the time being) once the adversarial total variation $\nu\mhyphen\tv$ is dualized.

\subsection{Preliminaries from Convex Analysis}
We briefly review some basics from convex analysis, which shall appear frequently in this paper. 
If $X$ is a Banach space with continuous dual $X^*$, and $f\colon X\to(-\infty,\infty]$ is convex, we define its \emph{subdifferential} at $x\in X$ as
\begin{align}
    \partial f(x) = \{x^*\in X^*\st f(x)+\langle x^*,y-x\rangle_{X^*\times X}\leq f(y)\quad\forall y\in X\},
\end{align}
where $\langle \cdot,\cdot\rangle_{X^*\times X}$ denotes the dual pairing of $X^*$ and $X$. The elements of $\p f(x)$ are called \emph{subgradients} of $f$.
Furthermore, we define the \emph{convex conjugate} $f^*\colon X^*\to\R$ of $f$ via
\begin{align}
    f^*(x^*) \coloneq \sup_{x\in X}\langle x^*,x\rangle_{X^*\times X} - f(x),\qquad x^*\in X^*,
\end{align}
which is a convex and lower semicontinuous function.
The two notions are connected through the Fenchel--Young inequality, see \cite[Proposition 5.1]{ekeland1999}:
\begin{align}
    \langle x^*,x\rangle_{X^*\times X} \leq f(x)+f^*(x^*)
    \text{ with equality if and only if }x^*\in\partial f(x).
\end{align}
The subdifferential $\partial f(x)$ is convex and weakly-* closed.
Furthermore, using the canonical embedding $X\hookrightarrow X^{**}\coloneq(X^*)^*$ the biconjugate $f^{**}\coloneq(f^*)^*$ satisfies $f^{**}\vert_{X}\leq f$ with equality if and only if $f$ is convex and lower semicontinuous.

\subsection{Main Results and Discussion}
\label{sec:main_results}

The purpose of this paper is to derive dualization formulas, similar to the one for the classical total variation, for the adversarial total variation in two different scenarios:
\begin{enumerate}
    \item First, we let the base space be $C_0(\X)$, the space of continuous functions on a proper metric space $\X$ vanishing at infinity.
    In this case, the subdifferential will lie in the dual space of that space which is the space of finite signed Radon measures on $\X$.
    This choice of base space is relevant for situations where adversarial training is performed over a parametrized subset of the space of continuous functions, see \cite{lunk2026explaining}.
    In this setting the reference measure $\nu$ has no bearing and it holds for all $u\in C_0(\X)$ that
    \begin{align}
        \nu\mhyphen\tv_\eps(u) =
        \tv_\eps(u) \coloneq
        \int_\X
        \frac{\max_{{\clball{\eps}{x}}}u-u(x)}{\eps}
        \d\vrho_0(x)
        +
        \int_\X
        \frac{u(x)-\min_{{\clball{\eps}{x}}}u}{\eps}
        \d\vrho_1(x),
    \end{align}
    where $\clball{\eps}{x}$ is the closed ball around $x\in\X$.
    Furthermore, the total variation can be dualized as follows
    \begin{align}
        \tv_\eps(u)=\max_{\Boldserifm\in{\MM_\eps}\times{\MM_\eps}}\left\lbrace
        -\int_\X u\dd \,\div_\eps[\Boldserifm;\Frho]\right\rbrace.
    \end{align}
    Here, $\Frho\coloneq(\rho_0,\rho_1)$ is the collection of the two class-conditional distributions, $\div_\eps[\Boldserifm;\Frho]$ is a finite signed measure, and $\div_\eps$ is a suitable notion of nonlocal divergence that is adjoint to a nonlocal gradient $\grad_\eps:C_0(\X)\to C_0(\X\times\X)$. 
    The set of test ``functions'' is a set of random walks ${\MM_\eps}$, i.e., a family of probability measures $\serifm=\{m_x\st x\in\X\}$ satisfying $\supp m_x \subset \clball{\eps}{x}$ as well as certain measurability conditions.
   
    Subgradients are the nonlocal divergences of random walks which attain the maximum above.

    These results are phrased in the language of metric random walk spaces and should be compared to the seemingly similar but fundamentally different setting in \cite{mazon2020}.

    See \cref{thm:u cont-dual representation,thm:u cont-subdiff characterization} for the rigorous statements.
    \item Second, we consider $L^\infty(\X)$ as base space where $\X$ is a separable metric space with a $\sigma$-finite reference measure $\nu$ such that $\vrho_0,\vrho_1\ll\nu$.
    The reason for considering this setting, as well, is that minimizers of \eqref{eq:reg problem} naturally lie in $L^\infty(\X)$, as proved in \cite{bungert2023geometry}.

    For $u\in L^\infty(\X)$ we have
    \begin{align}
        \nu\mhyphen\tv_\eps(u)
        &=
        \int_{\X}
        \frac{\nesssup{u}{\ball{\eps}{x}}-u(x)}{\eps}
        \dd\vrho_0(x)
        +
        \int_{\X}
        \frac{u(x) - \nessinf{u}{\ball{\eps}{x}}}{\eps}
        \dd\vrho_1(x)
        \\
        &=
        \sup_{\FPsi\in{\PP_\eps}\times{\PP_\eps}}\left\lbrace
        -\int_{\X} u\div_\eps\left[\FPsi;\Frho\right]\dd \nu(x)\right\rbrace.
    \end{align}
    While being structurally very similar to the previous dualization result, it is important to notice that here the test functions lie in ${\PP_\eps}\times {\PP_\eps}$ where ${\PP_\eps}$ is a suitable subset of $L^1(\X\times\X)$ and also $\div_\eps[\FPsi;\Frho]$ is a $L^1$-function.
    For notational convenience we overload the symbol $\div_\eps$ which has its own definition for measures and functions as arguments respectively.
    Note that the set of test functions is significantly smaller.
    Consequently, unlike in the first case, where the maximizing random walks are always attained thanks to the compactness implied by Prokhorov's theorem, here, due to a lack of compactness of this set of nonlocal divergences, the supremum is usually not attained. 
    
    Subgradients are maximizers lying in the weak-* closure of this set in the dual space of $L^\infty(\X)$ which is the space of bounded and finitely additive measures that are absolutely continuous with respect to $\nu$ which is a strictly bigger space than $L^1(\X)$, see \cite{dunford1988}.

    See \cref{thm:u bounded-dual representation,thm:u bounded-subdiff char} for the rigorous statements.
\end{enumerate}

We would like to remark that the two settings we consider in this paper are complementary to each other. 
For the base space $C_0(\X)$, which is separable for proper $\X$, the subdifferential lies in the nice space of signed measures and subgradients can be explicitly characterized as nonlocal divergences of random walks that attain the maximum in the dualization formula.
In contrast, for the larger space $L^\infty(\X)$, which is not separable, we cannot characterize subgradients explicitly.
Instead, in this case the dualization is taken over a nice set of $L^1$-functions on $\X\times\X$ instead of just random walks.

\subsection{Outline}
The rest of the paper is organized as follows. 
In \cref{sec:continuous} we derive a duality formulation of the adversarial total variation defined on the set of continuous functions vanishing at infinity, which subsequently leads to a characterization of its subdifferential. 
In \cref{sec:dualization_continuous} we first obtain dual representations of the maximum and minimum of a continuous function on $\eps$-balls around each data point and unify them into random walks on the whole domain via the measurable maximum theorem. Based on this reformulation, \cref{sec:grad_div_continuous} introduces a dualization involving a nonlocal gradient and divergence. 
Finally, in \cref{sec:subdiff_continuous} we characterize the subdifferential of $\tv_\eps$ in terms of nonlocal divergence measures arising from this construction.

\cref{sec:bounded} is devoted to the duality formulation of the adversarial total variation defined on essentially bounded functions on separable metric spaces. In \cref{sec:dualization_bounded} we derive an analogous duality formulation to that of \cref{sec:dualization_continuous}, which leads to a dualization for $\nu\mhyphen\tv_\eps$ in \cref{sec:grad_div_bounded}, again involving nonlocal gradient and divergence operators.
There we also show the consistency of these nonlocal operators with their classical local counterparts in the Euclidean setting. 
Lastly, in \cref{sec:subdiff_bounded} we provide a limit characterization of subgradients of the adversarial total variation in this setting.

We conclude the paper in \cref{sec:conclusions} with a brief discussion of possible directions for future research.

\section{Dualization for \texorpdfstring{$C_0(\X)$}{C0(X)}}\label{sec:continuous}

In this section, we provide a duality formulation for the nonlocal total variation functional, using $C_0(\X)$ for a proper metric space $\X$ as base space. In \cref{sec:dualization_continuous} we show that the set of test functions is given by random walks on metric spaces with additional support conditions, and in \cref{sec:grad_div_continuous} we define a nonlocal gradient and divergence such that an integration-by-parts identity holds, leading to the dual formulation of $\tv_\eps$. Finally, in \cref{sec:subdiff_continuous} we characterize the subgradients in $\partial\tv_\eps(u)$.

Throughout this section, we consider a continuous function $u\in C(\X)$ on a metric space $(\X,\mathrm{d})$, which is assumed to be proper unless stated otherwise. A metric space $\X$ is called \emph{proper} if every closed and bounded subset $K\subset\X$ is compact, for instance $\X=\overline{\Ome}\subset\IRN$. In particular, every proper metric space is Polish and locally compact.
Furthermore, $C_0(\X)$ refers to the subspace of functions in $C(\X)$ which vanish at infinity. 
For compact $\X$ it holds $C(\X)=C_0(\X)$.

In the following we assume that the reference measure $\nu$ satisfies the assumptions imposed in~\cite{bungert2023geometry}. In particular, $\nu$ is a $\sigma$-finite measure on $\X$ such that $\vrho_0,\vrho_1\ll\nu$ and the following support condition holds
	\begin{align}\label{eq:nu support}
		\set{x\in\X}{\mathrm{dist}(x,\supp\vrho_0)\leq \eps}\cup\set{x\in\X}{\mathrm{dist}(x,\supp\vrho_1)\leq \eps}\subseteq \supp\nu.
	\end{align}
Under these assumptions, $\nu\mhyphen\tv_\eps$ is independent of the particular choice of the reference measure, and we therefore omit its explicit dependence in the notation. To see this, note that we have $\nesssup{u}{{\clball{\eps}{x}}}=\max_{{\clball{\eps}{x}}} u$ for $\vrho_0$-a.e.\ $x\in\X$ and $\nessinf{u}{{\clball{\eps}{x}}}=\min_{{\clball{\eps}{x}}} u$ for $\vrho_1\mhyphen$a.e. $x\in \X$ since ${\clball{\eps}{x}}\subset\X$ is compact and $\nu$ is supported on the whole ball $\clball{\eps}{x}$.
Hence, in the whole of \cref{sec:continuous} we denote
\begin{align}
    \tv_\eps \coloneq \nu\mhyphen\tv_\eps
\end{align}
where $\nu\mhyphen\tv_\eps$ was defined in \eqref{eq:TV functional}. Furthermore we define $\vrho\coloneq \vrho_0+\vrho_1$, implying $\vrho_0,\vrho_1\ll \vrho$.
\subsection{Dual Reformulation for Continuous Functions}\label{sec:dualization_continuous}
To obtain a duality formulation, we first derive a dual representation of the maximum and minimum of a continuous function in a closed ball around a data point, which is achieved by testing with probability measures. 
Note that these quantities appear under the integrals in the definition of $\tv_\eps$, cf. \cref{sec:main_results}.
We then extract a measurable selector that defines a random walk on $\X$, i.e., a $\BB_\X$-measurable map from $\X$ to the set of probability measures on $\X$, where $\BB_\X$ denotes the Borel $\sigma$-algebra on $\X$. 

In this framework, we obtain a dual representation of both parts of the nonlocal total variation which we shall combine into a full dualization using a nonlocal and divergence in \cref{sec:grad_div_continuous} below. 

Let us briefly introduce the most important notions for the following statement. 
First, a random walk $\serifm=\set{m_x\in\P(\X)}{x\in\X}\in RW(\X)$ can be understood as a family of probability measures on $\X$ which depend measurably on $x$; the complete definition is given in \cref{def:random walk}. Next, the set of test functions is now defined as
\begin{align}
    {\MM_\eps}\coloneq \set{
    \serifm\in  
    RW(\X)
    }{\supp m_x \subset \clball{\eps}{x}\quad
    \text{for $\rho$-almost every }
    x\in\X
    }.
\end{align}
The following is the main theorem of this section and a key ingredient for dualizing the total variation.
\begin{theorem}[Dual reformulation]\label{thm:u cont-dual reformulation of tv}
    Let $\X$ be a proper metric space, let $\mu\in \P(\X\times\{0,1\})$ be a probability measure and $\vrho_i\coloneq \mu(\filll\times\{i\})$ for $i\in\{0,1\}$ the respective conditional distributions.
    Then, for  $u\in C(\X)$ we have
    \begin{align}
        \int_\X \max_{{\clball{\eps}{x}}}u \dd \vrho_0(x)=\max_{\serifm\in{\MM_\eps}}\int_\X\int_\X u(y)\dd m_x(y)\dd\vrho_0(x)
    \end{align}
    and
    \begin{align}
        \int_\X \min_{{\clball{\eps}{x}}}u \dd \vrho_1(x)=\min_{\serifm\in{\MM_\eps}}\int_\X\int_\X u(y)\dd m_x(y)\dd\vrho_1(x).
    \end{align}
\end{theorem}
\begin{proof}[Proof of \cref{thm:u cont-dual reformulation of tv}]
    The application of \cref{prop:u cont max dualisation} and \cref{lem:u cont-dual representation} below leads to the first equality.
    The second equality is obtained in an analogous way.
\end{proof}
As a first step, to obtain a dual formulation of the maximum of $u\in C(\X)$ over $\clball{\eps}{x}$, we use probability measures as test functions and note that the maximum is attained by Dirac measures concentrated at maximizers.
\begin{proposition}\label{prop:u cont max dualisation}
    Let $\X$ be a proper metric space. Then, for every $u\in C(\X)$ and $x\in\X$ we have
    \begin{align}
        \max_{{\clball{\eps}{x}}} u=\max_{m\in\P({\clball{\eps}{x}})}\int_{{\clball{\eps}{x}}} u \dd m.
    \end{align}
\end{proposition}
\begin{proof}
The result follows from the fact that $u$, being a continuous function on $\clball{\eps}{x}$, which is compact because of the properness of $\X$, attains its maximum in $\clball{\eps}{x}$.
It trivially holds
\begin{align}
    \max_{{\clball{\eps}{x}}} u\geq\sup_{m\in\P({\clball{\eps}{x}})}\int_{{\clball{\eps}{x}}} u \dd m.
\end{align}
Furthermore, the probability measure $m=\delta_{\hat x}\in\mathcal{P}(\clball{\eps}{x})$ where $\hat x\in\clball{\eps}{x}$ is chosen such that $u(\hat x)=\max_{\clball{\eps}{x}}u$ attains the supremum on the right hand side.

\end{proof}

Since the dual space of continuous functions consists of measures, we need to work with a suitable topology on that space. There are several ways to define convergence of measures. 
The following definition is based on the one from \cite[Chapter 13]{klenke2020} to finite signed measures which we denote by $\M^\pm_f(\X)$.
A signed measure $\mu$ is called a \emph{finite signed measure} if both its positive and negative parts $\mu^+$ and $\mu^-$, given by the Jordan decomposition, are finite measures.
\begin{definition}[Weak convergence of measures]\label{defi:weak measure convergence}
    Let $\X$ be a metric space and $m\in \M^\pm_f(\X)$. We say that $(m_n)_{n\in\IN}\subset\M_f^\pm(\X)$ \emph{converges weakly} to $m$ if
    \begin{align}
        \int_\X f\dd m_n\xrightarrow{n\to\infty}\int_\X f\dd m\qquad\text{for all }f\in C_b(\X),
    \end{align}
    where $C_b(\X)$ denotes the space of continuous and bounded functions on $\X$.
    The topology induced by weak convergence is called \emph{weak topology} $\tau_w$.
\end{definition}
Note that weak convergence is sometimes referred to as narrow convergence in the literature describing the same concept. To distinguish limits of measures from limits of functions, we write $m=\wlim{n\to\infty}m_n$ if $(m_n)_{n\in\IN}\subset \M_f^\pm(\X)$ converges weakly to a limit measure $m\in\M_f^\pm(\X)$.
\begin{remark}[Measure convergence from a functional analytic perspective]\label{rem:weak_topology}
Note that the term weak convergence may be slightly misleading, since $\tau_w$ does not correspond to the weak topology in the functional analytic sense (unless $\X$ is finite).
Instead, at least if $\X$ is compact, $\tau_w$ coincides with the weak-* topology on $\M_f^\pm(\X)$.
To remain consistent with the literature, we keep this terminology.
\end{remark}
The measures over which the maximum in \cref{prop:u cont max dualisation} is taken, and in particular the maximizing one, obviously depend on $x$.
Therefore, in order to obtain \cref{thm:u cont-dual reformulation of tv} from \cref{prop:u cont max dualisation}, we need to extract maximizing measures in a measurable way and group them into a random walk $\serifm=\{m_x\st m_x\in\mathcal{P}(\clball{\eps}{x})\}$ for which we shall require the theory of measurable selectors of correspondences.
In our case, the correspondence of interest will be the set-valued map $x\mapsto\mathcal{P}(\clball{\eps}{x})$.

We briefly introduce the most relevant definitions.
A \emph{correspondence} $\phi\colon S\corr X$ from a set $S$ to a set $X$ is a set-valued function that assigns to each $s$ in $S$ a subset $\phi(s)$ of $X$. For a correspondence $\phi$ and a subset $A\subset X$ the \emph{lower inverse} $\phi^\ell$ is given by $\phi^\ell(A)\coloneq \set{s\in S}{\phi(s)\cap A\neq \emptyset}$. A \emph{selector} from a correspondence $\phi$ is a function $f\colon S\to X$ that satisfies $f(s)\in\phi(s)$ for each $s\in S$. If $(S,\Sigma)$ is a measurable space and $X$ a topological space, we say that a correspondence $\phi\colon S\corr X$ is \emph{weakly measurable} if $\phi^\ell(G)\in \Sigma$ for each open subset $G\subset X$ and that $\phi$ is \emph{measurable} if $\phi^\ell(F)\in\Sigma$ for each closed subset $F\subset  X$. A complete introduction to this topic can be found in \cite{aliprantis2006, fonseca2007, rockafeller1998variational}.
\begin{lemma}\label{lem:u cont-weak measurable correspondence}
    Let $\X$ be a proper metric space and define the correspondence
    \begin{align}\label{eq:correspondence to measures}
        \psi\colon \X \corr \P(\X),\quad \psi(x)\coloneq\P({\clball{\eps}{x}}).
    \end{align}
    Then, the correspondence $\psi$ is weakly measurable with respect to $\BB_\X$.
\end{lemma}
\begin{proof}
    Let $G\subset \P(\X)$ be an open set with respect to $\tau_w$. We show that 
    \begin{align}
        \psi^\ell(G)=\set{x\in\X}{G\cap\psi(x)\neq \emptyset}=\set{x\in\X}{\exists m\in G\text{ s.t. }m\in\P({\clball{\eps}{x}})}
    \end{align}
    is open by proving that its complement is closed. Let $(x_n)_{n\in\IN}\subset \psi^\ell(G)^c$ and $x\in \X$ such that $x_n\to x$ in $\X$. Then, by definition, we have $G\cap\psi(x_n)=\emptyset$ for any $n\in \IN$ and we aim to show that the same is true for $x$, i.e., that for any $m\in \P(\clball{\eps}{x})$ it holds that $m\notin G$. 
    
    So, let $m\in\P(\clball{\eps}{x})$ and, for now, assume that $m(\partial \ball{\eps}{x})=0$. For $\eps>\eps'>0$ define
    \begin{align}
        m_{\eps'}\coloneq \frac{1}{m(\clball{\eps'}{x})}m\restr{\clball{\eps'}{x}}.
    \end{align}
    This is well-defined for $\eps'$ close to $\eps$, since by the continuity of measures from below~\cite[Theorem 1.36]{klenke2020} we have
    \begin{align}
        c_{\eps'}\coloneq m(\clball{\eps'}{x})\to m(\ball{\eps}{x})=1-m(\partial \ball{\eps}{x})=1
        \qquad
        \text{as }\eps'\uparrow\eps
    \end{align}
    and hence $c_{\eps'}>0$ for $\eps'$ close enough to $\eps$.
    Then, for $f\in C_b(\clball{\eps}{x})$ we have
    \begin{align}
        \int_{\clball{\eps}{x}}f\dd m-\int_{\clball{\eps}{x}} f \dd m_{\eps'}=\int_{\clball{\eps}{x}\setminus\clball{\eps'}{x}}f\dd m + \left(1-\frac{1}{c_{\eps'}}\right)\int_{\clball{\eps'}{x}}f \dd m,
    \end{align}
    hence
    \begin{align}
        \abs{\int_{\clball{\eps}{x}} f \dd m-\int_{\clball{\eps}{x}} f \dd m_{\eps'}}\leq 2\norm{f}_\infty (1-c_{\eps'})\xrightarrow[\eps'\uparrow\eps]{}0.
    \end{align}
    This shows that $m=\wlim{\eps'\uparrow\eps}m_{\eps'}$ by definition of the weak topology $\tau_w$ from \cref{defi:weak measure convergence}.
    
    Note that for fixed $\eps'$ there is an $N\in\IN$ such that $\dd(x_n,x)<\eps-\eps'$ for all $n\geq N$ due to the convergence of $(x_n)_{n\in\IN}$ to $x$. By the triangle inequality we infer that $\clball{\eps'}{x}\subset\clball{\eps}{x_n}$ for $n\geq N$ implying that $m_{\eps'}\in \P(\clball{\eps'}{x})\subset \P(\clball{\eps}{x_n})=\psi(x_n)$. However, since $G\cap\psi(x_n)=\emptyset$, it follows that $m_{\eps'}\in \P(\X)\setminus G$ which is closed with respect to $\tau_w$ implying that $m$, being the weak limit of $m_{\eps'}$, does not lie in $G$.
    
    It is now left to consider the case of $m(\partial \ball{\eps}{x})>0$. 
    Since $\X$ is proper, $\clball{\eps}{x}$ is compact and thus separable implying that there exists a countable dense subset. 
    In particular, there are $(\xi_n)_{n\in\IN}\subset\clball{\eps}{x}$ such that for any open and nonempty $A\subset \clball{\eps}{x}$ there is an $n\in\IN$ so that $\xi_n\in A$. We define $\mu\coloneq \sum_{i\in\IN} 2^{-i}\delta_{\xi_i}\in \P(\clball{\eps}{x})$ which satisfies $\mu(A)>0$ for any open and nonempty subset $A\subset\clball{\eps}{x}$. Next, fix a small $r>0$ and consider the trivial cover $\partial\ball{\eps}{x}\subset\bigcup_{z\in \partial\ball{\eps}{x}} \ball{r}{z}$. Due to the properness of $\X$, $\partial \ball{\eps}{x}$ is compact and thus there is a finite subcover
    \begin{align}
        \partial \ball{\eps}{x}\subset \bigcup_{i=1}^N \ball{r}{z_i}.
    \end{align}
    For $i=1,\ldots,N$ we define the pairwise disjoint sets 
    \begin{align}
        C_i\coloneq \ball{r}{z_i}\setminus \bigcup_{1\leq j<i} \ball{r}{z_j}
    \end{align}
    and the measure
    \begin{align}
        m_r\coloneq m\restr{\left(\clball{\eps}{x}\setminus \bigcup_{i=1}^N C_i\right)}+\sum_{i=1}^N m(C_i) \frac{\mu\restr{\left(\ball{r}{z_i}\cap\ball{\eps}{x}\right)}}{\mu(\ball{r}{z_i}\cap\ball{\eps}{x})}.
    \end{align}
    Note that 
    \begin{align}
        m_r(\clball{\eps}{x})=m\left(\clball{\eps}{x}\setminus \bigcup_{i=1}^N C_i\right)+\sum_{i=1}^N m(C_i)=1
    \end{align}
    and by construction $m(\partial \ball{\eps}{x})=0$. Hence, by the argumentation in the first case it follows that $m_r\notin G$. 
    Additionally, for $f\in C_b(\clball{\eps}{x})$ we have
    \begin{align}
        \int_{\clball{\eps}{x}} f\dd m- \int_{\clball{\eps}{x}}f \dd m_r&=\int_{\bigcup_{i=1}^N (C_i\cap\clball{\eps}{x})} f\dd m-\sum_{i=1}^N m(C_i)\mint_{\ball{r}{z_i}\cap\ball{\eps}{x}}f\dd\mu\\&=\sum_{i=1}^N\left[\int_{C_i\cap\clball{\eps}{x}} f\dd m- \int_{C_i\cap\clball{\eps}{x}}\mint_{\ball{r}{z_i}\cap\ball{\eps}{x}}f\dd\mu\dd m\right]
    \end{align}
    where in the last equation we used that $m$ is supported in $\clball{\eps}{x}$. Furthermore, 
    \begin{align}
        \int_{C_i\cap\clball{\eps}{x}}& \abs{f-\mint_{\ball{r}{z_i}\cap\ball{\eps}{x}} f \dd\mu} \dd m=\int_{C_i\cap\clball{\eps}{x}} \abs{\mint_{\ball{r}{z_i}\cap\ball{\eps}{x}} f(y)-f(y') \dd\mu(y')} \dd m(y)\\&\leq \int_{C_i\cap\clball{\eps}{x}} \mint_{\ball{r}{z_i}\cap\ball{\eps}{x}} \abs{f(y)-f(y')}\dd \mu(y')\dd m(y)\\&\leq m(C_i\cap\clball{\eps}{x}) \sup_{\substack{y\in C_i\cap\clball{\eps}{x}\\y'\in \ball{r}{z_i}\cap\ball{\eps}{x}}} \abs{f(y)-f(y')}
    \end{align}
    for $i=1,\ldots,N$.
    Combining the last two displays we arrive at
    \begin{align}
        \abs{\int_{\clball{\eps}{x}} f\dd m- \int_{\clball{\eps}{x}}f \dd m_r}&\leq \sum_{i=1}^N m(C_i\cap\clball{\eps}{x}) \sup_{y,y'\in \ball{r}{z_i}\cap\clball{\eps}{x}} \abs{f(y)-f(y')}\\&\leq \sum_{i=1}^N m(C_i\cap\clball{\eps}{x})\,\omega_f(2r)\leq \omega_f(2r)\xrightarrow[r\to 0]{}0
    \end{align}
    where $\omega_f$ is the modulus of continuity of $f$ which exists and satisfies $\omega_f(2r)\to 0$ for $r\to 0$ because $f$ is uniformly continuous on $\clball{\eps}{x}$. Hence, $m=\wlim{r\to 0}m_r$ and thus $m\notin G$. 
    
    So, in both cases $m\notin G$ and because $m\in \psi(x)$ was arbitrarily chose we obtain that $G\cap \psi(x)=\emptyset$. In particular, $x\in \psi^\ell(G)^c$ which is therefore sequentially closed in $\X$, implying that $\psi^\ell(G)$ is open and therefore a Borel set.
\end{proof}
We now establish the existence of a measurable selector for the correspondence $\psi$ from \cref{lem:u cont-weak measurable correspondence}. 
Moreover, we show that the maximum over the $x$-dependent test sets is measurable in $x$, and the set of maximizers is nonempty and compact.
\begin{proposition}\label{prop:u cont-measurable selector}
    Let $\X$ be a proper metric space and $u\in C(\X)$. Define $u^\eps\colon \X\to\IR$ by
    \begin{align}\label{eq:u cont-max problem}
        u^\eps(x)=\max_{m\in\P({\clball{\eps}{x}})}\int_{{\clball{\eps}{x}}}u(y)\dd m(y).
    \end{align}
    and the argmax set
    \begin{align}
        A^x\coloneq \set{m\in\P({\clball{\eps}{x}})}{\int_{{\clball{\eps}{x}}} u(y)\dd m(y)=u^\eps(x)}.
    \end{align}
    Then,
    \begin{enumerate}[label=(\roman*)]
        \item $u^\eps$ is measurable,
        \item $A^x$ is nonempty and compact in the $\tau_w$-topology for any $x\in\X$, and
        \item there exists a measurable selector function $\serifPsi\colon \X\to \P(\X)$, meaning that $\serifPsi(x)\in A^x$ for any $x\in\X$ and $\serifPsi$ is measurable, i.e., for any $\tau_w$-open subset $G\subset\P(\X)$ the preimage $\serifPsi^{-1}(G)$ lies in the Borel $\sigma$-algebra on $\X$.
    \end{enumerate}
\end{proposition}
\begin{proof}
Since $\X$ is proper, ${\clball{\eps}{x}}\subset\X$ is compact and thus due to Prokhorov's theorem, $\P({\clball{\eps}{x}})$ is weakly sequentially compact. Hence, the correspondence $\psi$ as defined in \eqref{eq:correspondence to measures}, which is weakly measurable thanks to \cref{lem:u cont-weak measurable correspondence}, also has nonempty and compact values and we aim to apply the classical measurable maximum theorem, see \cite[Theorem 18.19]{aliprantis2006}. To this end we define $f\colon \Gr\psi\to\IR, f(x,m)=\int_{{\clball{\eps}{x}}} u(y)\dd m(y)$ and by the fact that $\Ie_{{\clball{\eps}{x}}}(y)$ is jointly measurable we obtain that $f$ is a Carathéodory function. Lastly, since $\X$ is, a fortiori, Polish and locally compact, the space $(\P(\X),\tau_w)$ is Polish. So, the measurable maximum theorem is indeed applicable and completes the proof. 
\end{proof}
Using the measurability of a measurable selector, we interpret the family of measures provided by the selector function as a random walk, also known as Markov kernels in the literature. More precisely, for every $x\in\X$ we obtain a probability measure supported in the $\eps$-ball around $x$, describing the probability distribution of a single step of the walk. We introduce random walks following \cite{mazon2020}.
\begin{definition}[Random walk]\label{def:random walk}
    Let $(\X,\dd)$ be a Polish metric space equipped with its Borel $\sigma$-algebra $\BB_\X$. A \emph{random walk} $\serifm$ on $\X$ is a family of probability measures $m_x$ on $\X$ for $x\in \X$ satisfying
    \begin{enumerate}[label=(\roman*)]
        \item[(i)] the measures $m_x$ depend measurably on the point $x\in \X$, i.e., for any Borel set $A$ of $\X$ the function $x\mapsto m_x(A)$ is Borel measurable;
        \item[(ii)] each measure $m_x$ has finite first moment, i.e., for some $x_0\in\X$ and for any $x\in \X$ one has $\int_\X\dd (x_0,y)\dd m_x(y)<+\infty$.
    \end{enumerate}
    We denote the set of random walks on $\X$ as ${RW}(\X)$. 
\end{definition}
Next, we show that random walks induce measurable parameter integrals. 
This is particularly useful since it allows integration without explicitly invoking the Fubini--Tonelli theorem.
\begin{lemma}\label{lem:random walks measurable}
    Let $\X$ be a Polish metric space and $u\colon \X\to\IR$ a measurable function. Then, for $\serifm\in {RW}(\X)$ the mapping
    \begin{align}
        I\colon \X\to\IR,\quad x\mapsto \int_\X u\dd m_x
    \end{align}
    is measurable.
\end{lemma}
\begin{proof}
    First, we show the statement for $u$ being a simple function. By definition there is a family of Borel sets $(A_i)\subset \X$ and $\alpha_i\in\IR\setminus\{0\}$ ($i\in\{1,\ldots,n\}$) such that $u=\sum_{i=1}^n \alpha_i \Ie_{A_i}$. Define
    \begin{align}
        I_i\colon \X\to\IR,\quad  x\mapsto \int_\X\alpha_i\Ie_{A_i}\dd m_x=\alpha_i m_x(A_i)
    \end{align}
    which is a measurable mapping since the righthand side is measurable due to the definition of a random walk. Linear combinations of measurable functions still being measurable yields that
    \begin{align}
        I(x)=\int_\X u\dd m_x=\int_\X  \sum_{i=1}^n \alpha_i \Ie_{A_i} \dd m_x=\sum_{i=1}^n I_i(x)
    \end{align}
    is measurable.
    
    Now assume that $u\colon \X\to\IR$ is a general measurable function. We can split $u=u^+-u^-$ into the difference of two non-negative measurable functions. For each of them there is a sequence of simple functions $(u_n^+)_{n\in\IN}$ and $(u_n^-)_{n\in\IN}$ such that $u_n^+(y)\uparrow u^+(y)$ and $u_n^-(y)\uparrow u^-(y)$ for all $y\in\X$, see \cite[Theorem 4.36]{aliprantis2006}. This implies that
    \begin{align}
        I(x)&=\int_\X u(y) \dd m_x(y)=\int_\X u^+(y)\dd m_x(y)-\int_\X u^-(y)\dd m_x(y)\\
        &=\int_\X \lim_{n\to\infty} u_n^+(y)\dd m_x(y)-\int_\X \lim_{n\to\infty}u_n^-(y)\dd m_x(y)\\
        &=\lim_{n\to\infty}\int_\X u_n^+(y)\dd m_x(y)-\lim_{n\to\infty}\int_\X u_n^-(y)\dd m_x(y)
    \end{align}
    where we used the monotone convergence theorem for the last equation. Due to the first part of the proof and the fact that the pointwise limit of measurable functions is itself measurable we obtain that $I$ is the difference of two measurable functions and thus a measurable function itself.
\end{proof}
Next, we prove that integrals over the maximization problems in \cref{prop:u cont max dualisation} can be simplified by using random walks, thereby obtaining a set of test functions independent of $x$.
This is the last missing ingredient to complete the proof of \cref{thm:u cont-dual reformulation of tv}.
\begin{lemma}\label{lem:u cont-dual representation}
    Let $\X$ be a proper metric space and define
    \begin{align}
        {\MM_\eps}\coloneq \set{\serifm\in{RW}(\X)}{\supp m_x\subset{\clball{\eps}{x}}\text{ for }\vrho\text{-almost every }x\in\X}.
    \end{align}
    Then for every $u\in C(\X)$ and $i\in\{0,1\}$ we have
    \begin{align}
        \int_\X \max_{m\in\P({\clball{\eps}{x}})} \int_{{\clball{\eps}{x}}} u(y)\dd m(y) \dd \vrho_i(x)=\max_{\serifm\in{\MM_\eps}}\int_\X\int_\X u(y)\dd m_x(y)\dd\vrho_i(x).
    \end{align}
\end{lemma}
\begin{proof}
    For the first inequality let $\serifm\in {\MM_\eps}$. Then, by definition for $\vrho$-a.e. $x\in\X$ the measure $m_x\restr {{\clball{\eps}{x}}}\in\P({\clball{\eps}{x}})$ and thus
    \begin{align}
        \max_{m\in\P({\clball{\eps}{x}})} \int_{{\clball{\eps}{x}}} u(y)\dd m(y)\geq \int_{{\clball{\eps}{x}}} u(y) \dd (m_x\restr {{\clball{\eps}{x}}})(y)=\int_\X u(y)\dd m_x(y).
    \end{align}
    By \cref{prop:u cont-measurable selector}~(i) the lefthand side is measurable on $\X$ and by \cref{lem:random walks measurable} the same is true for the righthand side. Hence, integration over $\X$ yields
    \begin{align}
            \int_\X \max_{m\in\P({\clball{\eps}{x}})} \int_{{\clball{\eps}{x}}} u(y)\dd m(y) \dd \vrho_i(x)\geq \int_\X\int_\X u(y)\dd m_x(y)\dd\vrho_i(x).
    \end{align}
    Lastly, since $\serifm\in{\MM_\eps}$ was arbitrarily chosen we obtain
    \begin{align}
        \int_\X \max_{m\in\P({\clball{\eps}{x}})} \int_{{\clball{\eps}{x}}} u(y)\dd m(y) \dd \vrho_i(x)\geq \max_{\serifm\in{\MM_\eps}}\int_\X\int_\X u(y)\dd m_x(y)\dd\vrho_i(x).
    \end{align}
    For the converse inequality we define a random walk by using the selector $\serifPsi\colon \X\to\P(\X)$ which we obtain from \cref{prop:u cont-measurable selector}. Define $\hat{\serifm}=(\hat{m}_x)$ where $\hat{m}_x$ is the extension of $\serifPsi(x)$ on $\X$ by zero for every $x\in\X$. Then, this family of probability measures on $\X$ is measurably dependent on $x\in\X$ due to the measurability of $\serifPsi$. Furthermore, it has also finite first moments which means that $\hat{\serifm}\in RW(\X)$. 
    Due to the construction, the support condition is also satisfied meaning that $\hat{\serifm}\in{\MM_\eps}$. Apart from that, the properties of the selector $\serifPsi$ imply that
    \begin{align}
        \int_\X u(y)\dd\hat{m}_x(y)=\max_{m\in\P({\clball{\eps}{x}})}\int_{{\clball{\eps}{x}}} u \dd m.
    \end{align} 
    By \cref{lem:random walks measurable} the lefthand side is measurable, whereas the measurability of the righthand side is given by \cref{prop:u cont-measurable selector}. 
    So, we can once again integrate over $\X$ to obtain
    \begin{align}
        \max_{\serifm\in{\MM_\eps}}&\int_\X\int_\X u(y)\dd m_x(y)\dd\vrho_i(x)\\&\geq \int_\X\int_\X u(y)\dd \hat{m}_x(y)\dd \vrho_i(x)=\int_\X \max_{m\in\P({\clball{\eps}{x}})} \int_{{\clball{\eps}{x}}} u(y)\dd m(y) \dd \vrho_i(x)
    \end{align}
    which finishes the proof.
\end{proof}
\subsection{Nonlocal Gradient and Divergence}\label{sec:grad_div_continuous}
Based on the dual reformulation of each part of the nonlocal total variation, we rearrange the terms and obtain suitable notions of nonlocal gradient and nonlocal divergence operators connected by an integration-by-parts identity. 

Before stating our main result, we define its essential ingredients; the nonlocal gradient and divergence.
\begin{definition}[Nonlocal gradient]\label{def:nonlocal_gradient_continuous}
    For $\eps>0$ the nonlocal gradient $\grad_\eps\colon C(\X)\to C(\X\times\X)$ is defined by
    \begin{align}
        \grad_\eps[u](x,y) \coloneq \frac{u(y)-u(x)}{\eps},\qquad x,y\in\X,
    \end{align}
    for $u\in C(\X)$.
\end{definition}
Next, we define a nonlocal divergence which acts as the negative adjoint operator of the nonlocal gradient and maps finite signed measures on $\X\times\X$ to such on $\X$.
We denote the projection onto the first and second argument by $\pi_1(x,y)\coloneq x$ and $\pi_2(x,y)\coloneq y$ respectively for $x,y\in\X$.
\begin{definition}[Nonlocal divergence]\label{def:nonlocal_divergence_random_walk}
    For $\eps>0$ the nonlocal divergence $\div_\eps\colon \M^\pm_f(\X\times\X)\to\M_f^\pm(\X)$ is defined by
    \begin{align}
        \div_\eps[\mu]\coloneq \frac{(\pi_1)_\sharp\mu-(\pi_2)_\sharp\mu}{\eps}
    \end{align}
    for $\mu\in\M_f^\pm(\X\times\X)$.
\end{definition}
\begin{remark}
    For a Borel set $A\in\BB_\X$ and $\mu\in \M_f^\pm(\X)$ the definition of the nonlocal divergence measure reads
    \begin{align}\label{u cont-div explicetly}
        \div_\eps[\mu](A)\coloneq \frac{\mu(A\times\X)-\mu(\X\times A)}{\eps}.
    \end{align}
\end{remark}
Given a pair of random walks $\Boldserifm=(\serifm^0,\serifm^1)\in RW(\X)\times RW(\X)$ and the pair of conditional distributions $\Frho=(\vrho_0,\vrho_1)$ there exists a unique probability measure on the product space $\X\times\X$ equipped with the Borel $\sigma$-algebra $\BB_\X\otimes\BB_\X$, called the semidirect product $\vrho_i\ltimes \serifm^i$ for $i\in\{0,1\}$ that satisfies
\begin{align}\label{eq:u cont-semidirect measure on rectangles}
    (\vrho_i\ltimes\serifm^i)(A\times B) = \int_A m_x^i(B)\dd\vrho_i(x)\qquad\text{for }A,B\in\BB_\X
\end{align}
and 
\begin{align}\label{eq:u cont-semidirect measure}
\int_{\X\times\X}f(x,y)\dd(\vrho_i\ltimes\serifm^i)(x,y)=\int_\X\int_\X f(x,y) \dd m_x^i(y)\dd\vrho_i(x)
\end{align}
for any jointly measurable function $f$. For further details see \cite[Theorem 6.11]{ccinlar2011probability}. Also note that sometimes the semidirect is written using $\otimes$ instead of $\ltimes$, however, to clearly distinguish it from the product measure we choose the symbol used in algebraic literature for a semidirect product. We then define the antisymmetric pairing of $\Boldserifm$ and $\Frho$ by 
\begin{align}\label{eq:u cont-pairing def}
    [\Boldserifm;\Frho]\coloneq (\vrho_0\ltimes \serifm^0)-(\vrho_1\ltimes \serifm^1)\in \M_f^\pm(\X\times\X). 
\end{align}
The following theorem consists of two equations that are stated and proved below in two separate statements.

\begin{theorem}[Dual representation of $\tv_\eps$]\label{thm:u cont-dual representation}
    Let $\X$ be a proper metric space. Furthermore, let $\mu\in \P(\X\times\{0,1\})$ be a probability measure and $\vrho_i\coloneq \mu(\filll\times\{i\})$ for $i\in\{0,1\}$ the respective conditional distributions.
    For $u\in C(\X)$ it holds
    \begin{align}
        \tv_\eps(u)
        =
        \max_{\Boldserifm\in{\MM_\eps}\times{\MM_\eps}}\int_{\X\times\X} \grad_\eps[u] \dd [\Boldserifm;\Frho]
        =\max_{\Boldserifm\in{\MM_\eps}\times{\MM_\eps}}-\int_\X u\dd \,\div_\eps[\Boldserifm;\Frho].
    \end{align}
\end{theorem}
\begin{proof}
    The first equality is proved in \cref{prop:u cont-nonlocal gradient}, the second one in \cref{cor:u cont-nonlocal divergence} below.
\end{proof}
We start by stating and proving the first equality in the theorem.
\begin{proposition}\label{prop:u cont-nonlocal gradient}
    Under the conditions of \cref{thm:u cont-dual representation} for all $u\in C(\X)$ we have
    \begin{align}
        \tv_\eps(u)
        =
        \max_{\Boldserifm\in{\MM_\eps}\times{\MM_\eps}}\int_{\X\times\X} \grad_\eps[u] \dd [\Boldserifm;\Frho].
    \end{align}
\end{proposition}
\begin{proof}
    Due to \cref{thm:u cont-dual reformulation of tv},the properties of $\serifm\in {\MM_\eps}$ and \eqref{eq:u cont-semidirect measure} we have
    \begin{align}
        \frac{1}{\eps}\int_\X\max_{{\clball{\eps}{x}}}u-u(x)\dd \vrho_0(x)
        &=
        \frac{1}{\eps}\left[\max_{\serifm^0\in{\MM_\eps}}\int_\X\int_\X u(y) \dd m_x^0(y)\dd\vrho_0(x)-\int_\X u(x)\dd \vrho_0(x)\right]\\
        &=\max_{\serifm^0\in{\MM_\eps}}\frac{1}{\eps}\left[\int_\X
        \left(\int_\X u(y)\dd m^0_x(y)-\underbrace{\int_\X\dd m^0_x(y)}_{=1} u(x)\right)\dd\vrho_0(x)\right]\\
        &=\max_{\serifm^0\in{\MM_\eps}}\int_\X\int_{\X} \frac{u(y)- u(x)}{\eps}\dd m^0_x(y)\dd\vrho_0(x)\\
        &=\max_{\serifm^0\in{\MM_\eps}}\int_{\X\times\X} \grad_\eps[u](x,y)\dd (\vrho_0\ltimes \serifm^0)(x,y)
    \end{align}
    and similarly
    \begin{align}
        \frac{1}{\eps}\int_\X u(x)&-\min_{{\clball{\eps}{x}}}u\dd \vrho_1(x)=\max_{\serifm^1\in{\MM_\eps}}\int_{\X\times\X}-\grad_\eps[u](x,y)\dd (\vrho_1\ltimes\serifm^1)(x,y).
    \end{align}
    Summing both expressions, taking the maximum over $\serifm^0,\serifm^1\in{\MM_\eps}$, and using the definition of the antisymmetric pairing in \eqref{eq:u cont-pairing def} proves the statement.    
\end{proof}
The nonlocal operators $\grad_\eps$ and $\div_\eps$ satisfy the following integration-by-parts formula.
\begin{proposition}[Integration-by-parts]\label{prop:u cont-integration by parts}
    Let $\X$ be a metric space.
    For $u\in C(\X)$ and a finite signed measure $\mu\in\M_f^\pm(\X\times\X)$ we have
    \begin{align}
        \int_{\X\times\X}\grad_\eps[u]\d\mu = -\int_\X u \d\div_\eps[\mu].
    \end{align}
\end{proposition}
\begin{proof}
    Let $u\in C(\X)$ and $\mu\in \M_f^\pm(\X\times\X)$. Then,
    \begin{align}
        \int_{\X\times\X} \grad_\eps[u]\dd\mu&=\int_{\X\times\X}\frac{u(\pi_2(x,y))-u(\pi_1(x,y))}{\eps}\dd \mu(x,y)\\&=\frac{1}{\eps}\left(\int_\X u\dd (\pi_2)_\sharp\mu-\int_\X u \dd(\pi_1)_\sharp\mu\right)=-\int_\X u\dd\div_\eps[\mu]
    \end{align}
    where $\pi_1$ and $\pi_2$ indicate the projections on the first and second argument respectively.
\end{proof}
Based on that identity we can deduce the second equation of \cref{thm:u cont-dual representation}.
\begin{corollary}\label{cor:u cont-nonlocal divergence}
    Under the conditions of \cref{thm:u cont-dual representation} for all $u\in C(\X)$ and $\serifm\in {\MM_\eps}$ the following identity holds true for $i\in\{0,1\}$: 
    \begin{align}\label{eq:int-by-parts}
        \int_\X\int_\X
        \grad_\eps[u](x,y)
        \dd m_x(y)
        \dd\rho_i(x)
        =
        -\int_\X 
        u(y)\dd\div_\eps[\vrho_i\ltimes\serifm](y).
    \end{align}
    Moreover, we have
    \begin{align}
        \tv_\eps(u)
        =
        \max_{\Boldserifm\in{\MM_\eps}\times{\MM_\eps}} 
        -\int_\X 
        u\dd \div_\eps[\Boldserifm;\Frho].
    \end{align}
\end{corollary}
\begin{proof}
    Applying \cref{prop:u cont-integration by parts} to the semidirect product measure $\vrho_i\ltimes\serifm$ together with \eqref{eq:u cont-semidirect measure} yields \eqref{eq:int-by-parts}. In particular, by \cref{prop:u cont-nonlocal gradient} and \eqref{eq:int-by-parts} the second equality follows.
\end{proof}
Finally, we will prove that the divergence $\div_\eps$ from \cref{def:nonlocal_divergence_random_walk} is actually the unique probability measure which allows for an integration-by-parts formula like \eqref{eq:int-by-parts} and the adjoint operator to $\grad_\eps$.
Since we use duality techniques to show this, we restrict ourselves to $u\in C_0(\X)$, i.e., the space of continuous functions vanishing at infinity, the dual space of which coincides with the space of finite signed Radon measures.
Recall that a real function $u$ on a locally compact Hausdorff space $\X$ is said to \emph{vanish at infinity} if for every $\delta>0$ there exists a compact set $K\subset\X$ such that $\abs{f(x)}<\delta$ for all $x\in\X\setminus K$. 
Note that it holds $C_0(\X)\subset C_b(\X)$.
\begin{definition}\label{def:u cont-nonlocal divergence}
    Under the conditions of \cref{thm:u cont-dual representation}, we call a finite signed Radon measure $\gamma$ a nonlocal divergence of a product measure $\mu\in \M^\pm_f(\X\times\X)$ if
    \begin{align}\label{eq:divergence equation}
        -\int_\X u
        \dd\gamma
        =
        \int_{\X\times\X} \grad_\eps[u] \dd \mu
        \qquad
        \text{for all } u\in C_0(\X).
    \end{align}
\end{definition}
By applying the Riesz--Markov--Kakutani theorem, we show that \eqref{eq:divergence equation} admits a unique solution, given by our nonlocal divergence $\div_\eps$.
\begin{proposition}\label{prop:u cont-nonlocal divergence solution}
    For each $\mu\in \M_f^\pm(\X\times\X)$ there is a unique solution to \eqref{eq:divergence equation} which is given by the finite signed Radon measure $\div_\eps[\mu]$ defined in \cref{def:nonlocal_divergence_random_walk}. In particular, the operator $-\div_\eps:\M_f^\pm(\X\times\X)\to\M_f^\pm(\X)$ is  adjoint to $\grad_\eps:C_0(\X)\to C_0(\X\times\X)$.
\end{proposition}
\begin{proof}
     First, for $\mu\in \M^\pm_f(\X\times\X)$ define the map $G_{\mu}\colon C_0(\X)\to\IR$ by
    \begin{align}
        G_{\mu}[u]\coloneq -\int_{\X\times\X} 
        \grad_\eps[u]
        \dd \mu.
    \end{align}
    Since the nonlocal gradient and integrals are linear maps, the same is true for $G_{\mu}$. Next, we show that $G_{\mu}$ is a bounded map which will imply that $G_{\mu}\in C_0(\X)^*$.
    For any $u\in C_0(\X)$ we have
    \begin{align}
        \abs{G_{\mu}[u]}
        \leq
        \int_{\X}
        \abs{\frac{u(y)-u(x)}{\eps}}
        \dd \abs{\mu}(x,y)
        \leq 
        \frac{2\norm{\mu}\norm{u}_\infty }{\eps}
    \end{align}
    which shows that $G_{\mu}$ is indeed bounded.
    The application of the Riesz--Markov--Kakutani theorem (see, e.g., \cite[Theorem 1.54]{ambrosio2000functionsofBV}) yields the existence of a uniquely determined finite signed Radon measure $\gamma^*$ on $\X$ such that $G_{\mu}[u]=\int_\X u\dd\gamma^*$.
    The fact that $\gamma^*=-\div_\eps[\mu]$ follows directly from identity \cref{prop:u cont-integration by parts}.
\end{proof}

\subsection{Integral Characterization of the Subdifferential}\label{sec:subdiff_continuous}

Finally, we provide an integral characterization of the subdifferential of the adversarial total variation. Following \cite{bredies2016pointwise}, we use the dual formulation of the functional from \cref{thm:u cont-dual representation} to identify the subgradients. In our setting, we have to show that the set of nonlocal divergences corresponding to admissible random walks is convex and closed with respect to weak convergence of measures.
\begin{theorem}[Subdifferential characterization]\label{thm:u cont-subdiff characterization}
    Under the conditions of \cref{thm:u cont-dual representation} and assuming that $\X$ is compact, let $u\in C_0(\X)$.
    Then, $\mu^*\in \p\tv_\eps(u)$ if and only if 
    \begin{align}
        \begin{dcases}
            &\text{there exists a pair of random walks } \Boldserifm \in RW(\X)\times RW(\X)
            \\
            &\text{such that } \supp m_x^i\subset{\clball{\eps}{x}}\text{ for }\vrho\mhyphen\text{a.e. } x\in\X\text{ and }i\in\{0,1\},
            \\ 
            &\mu^*=-\div_\eps[\Boldserifm;\Frho]\text{ and }
            \tv_\eps(u)=\int_\X u\dd\mu^*.
        \end{dcases}
    \end{align}
\end{theorem}
Before we prove this result we need two lemmas that are concerned with a suitable notion of compactness for random walks.
\begin{lemma}\label{lem:measure support}
    Let $X$ be a metric space, $F\subset X$ be open, and $p\in\M_f(X)$ a finite measure such that $\int_X f\dd p=0$ for all $f\in C_b(X)$ with $f\equiv 0$ in $X\setminus F$. 
    Then, $\supp p \subset X\setminus F$.
\end{lemma}
\begin{proof}
    Aiming for a contradiction, assume that $\supp p \not\subset X\setminus F$ which implies that there exists $z\in F$ such that every open neighborhood of $z$ has positive measure. 
    Let $O\subset F$ be an open neighborhood of $z$.
    By applying Urysohn's Lemma we obtain a continuous function $f\colon \X\to[0,1]$ such that $f(z)=1$ and $f\equiv 0$ on $X\setminus O$, in particular, $f\equiv 0$ in $X\setminus F$. 
    Using $f(z)=1$ and the continuity of $f$, we can find a neighborhood $\tilde{O}\subset O$ of $z$ such that $f>0$ on $\tilde{O}$.
    As $p(\tilde{O})>0$, we get
    \begin{align}
        0<\int_{\tilde{O}} f \dd p\leq \int_X f\dd p=0
    \end{align}
    which is a contradiction. Hence, $\supp p\subset X\setminus F$.
\end{proof}
For using it in the proof of the following lemma and later on we define the open strip of width $\eps>0$ around the diagonal in $\X\times\X$ as
\begin{align}\label{eq:diagonal}        \Deps\coloneq\set{(x,y)\in\X\times\X}{\dd(x,y)<\eps}.
    \end{align}
\begin{lemma}\label{lem:compactness_RW}
    Let $\X$ be compact, $\rho\in\M_f(\X)$ be a finite measure, and $(\serifm_n)_{n\in\IN}\subset RW(\X)$. 
    Then, there exists $\serifm\in RW(\X)$ and a subsequence $(\serifm_{n_k})_{k\in\IN}\subset(\serifm_n)_{n\in\IN}$ such that
    \begin{align}
        \lim_{k\to\infty}
        \int_\X
        \int_\X
        g(x,y)\d m_{n_k,x}(y)\d\rho(x)
        =
        \int_\X
        \int_\X
        g(x,y)\d m_{x}(y)\d\rho(x)
        \qquad
        \text{for all } g\in C(\X\times\X).
    \end{align}
    If, moreover, $\supp m_{n,x}\subset\clball{\eps}{x}$ for $\rho$-almost every $x\in\X$ and all $n\in\IN$, then $\supp m_{x}\subset\clball{\eps}{x}$ for $\rho$-almost every $x\in\X$.
\end{lemma}
\begin{proof}
    If $\vrho(\X)=0$, then $\vrho=0$ and the statement holds trivially. 
    Hence, without loss of generality, assume that $\vrho$ is a probability measure, otherwise we replace $\vrho$ by $\frac{1}{\vrho(\X)}\vrho$.
    For each $n\in\IN$ there exists the semidirect product,  see \cite[Theorem 6.11]{ccinlar2011probability}, that is a unique probability measure $p_n$ on the product space $\X\times\X$ equipped with the product Borel $\sigma$-algebra $\BB_\X\otimes\BB_\X$ satisfying
    \begin{align}\label{eq:semiproduct def}
        p_n(A\times B)=\int_A m_{n,x}(B)\dd\vrho(x)\qquad\text{for }A,B\in \BB_\X,
    \end{align}
    and
    \begin{align}
        \int_{\X\times\X} f(x,y)\dd p_n(x,y)=\int_\X \int_\X f(x,y)\dd m_{n,x}(y)\dd\vrho(x)
    \end{align}
    for any $\BB_\X\otimes\BB_\X$-measurable function $f\colon \X\times\X\to\IR$. The latter is a direct consequence of \cite[Theorem 6.11]{ccinlar2011probability} by splitting $f$ in positive and negative part.
    Since $\X$ is compact, so is $\X\times\X$ and we can apply Prokhorov's theorem to obtain a subsequence $(p_{n_k})_{k\in\IN}\subset(p_n)_{n\in\IN}$ and $p\in\P(\X\times\X)$ such that $p=\wlim{k\to\infty}p_{n_k}$. Next, define $q\coloneq (\pi_1)_\sharp p$ where $\pi_1$ is the projection on the first argument. The application of the disintegration theorem, see \cite[Theorem 2.28]{ambrosio2000functionsofBV}, yields the existence of a $q$-almost everywhere uniquely determined family of probability measures $\{m_x\}_{x\in\X}$ satisfying the properties of a random walk on $\X$ and
    \begin{align}
        p(A\times B)=\int_A m_x(B)\dd q(x)\qquad \text{for }A,B\in\BB_\X.
    \end{align}
    Furthermore, for any measurable $f\colon \X\times\X\to\IR$ we have
    \begin{align}
        \int_{\X\times\X} f(x,y)\dd p(x,y)=\int_\X \int_\X f(x,y)\dd m_{x}(y)\dd q(x)
    \end{align}
    Next, for $n\in\IN$ define $q_n\coloneq (\pi_1)_\sharp p_n$ and note that for any $A\in\BB_\X$ we have
    \begin{align}
        q_n(A)=(\pi_1)_\sharp p_n(A)=p_n(A\times\X)=\int_A m_{n,x}(\X)\dd\vrho(x)=\vrho(A)
    \end{align}
    and hence trivially $\vrho=\wlim{k\to\infty}{q_{n_k}}$.
    Furthermore, since for any $f\in C(\X)$ the function $f\circ\pi_1$ lies in $C(\X\times\X)$, it holds that
    \begin{align}
        \int_\X f(x)\dd q(x)&=
        \int_{\X\times\X} f(\pi_1(x,y))\dd p(x,y)=
        \lim_{k\to\infty} \int_{\X\times\X} f(\pi_1(x,y))\dd p_{n_k}(x,y)\\&=\lim_{k\to\infty} \int_\X f(x)\dd q_{n_k}(x)
    \end{align}
    which implies $q=\wlim{k\to\infty}q_{n_k}$.
    So, by the uniqueness of the weak limit we have $q=\vrho$ and, in particular,
    \begin{align}
        p(A\times B)=\int_A m_x(B) \dd\vrho(x) \qquad \text{for }A,B\in\BB_\X.
    \end{align}
    Lastly, for any $g\in C(\X\times\X)$ we obtain
    \begin{align} 
        \lim_{k\to\infty} \int_\X\int_\X g(x,y)\dd m_{n_k,x}(y)\dd\vrho(x)
        &= 
        \lim_{k\to\infty} \int_{\X\times\X} g(x,y)\dd p_{n_k}(x,y)
        \\
        &=
        \int_{\X\times\X} g(x,y)\dd p(x,y)=\int_\X\int_\X g(x,y)\dd m_x(y)\dd\vrho(x).
    \end{align}
    For the second statement assume that $\supp m_{n,x}\subset \clball{\eps}{x}$ for $\vrho$-almost every $x\in\X$ and for all $n\in\IN$. Furthermore, define $\tilde{p}_n\colon \BB_\X\otimes\BB_\X\to\IR$ by
    \begin{align}
        \tilde{p}_n(E)\coloneq \int_\X m_{n,x}(E_x)\dd\vrho(x)\qquad\text{for }E\in \BB_\X\otimes\BB_\X
    \end{align}
    where $E_x\coloneq \set{y\in\X}{(x,y)\in E}$. Note that for two disjoint sets $E^1,E^2\subset \X\times\X$ we have that $(E^1\cup E^2)_x=E^1_x\cup E^2_x$ and thus $\tilde{p}_n$ defines a finite measure on $(\X\times\X,\BB_\X\otimes\BB_\X)$. Additionally, for $E=A\times B$, where $A,B\in\BB_\X$, we have that $E_x=B$ for $x\in A$ and $E_x=\emptyset$ otherwise and hence
    \begin{align}
        \tilde{p}_n(A\times B)=\int_\A m_{n,x}(B)\dd\vrho(X)=p_n(A\times B).
    \end{align}
    Since $p_n$ is the unique measure satisfying \eqref{eq:semiproduct def}, we obtain that $p_n=\tilde{p}_n$.
    Next, note that for the complement of $\clDeps$, where $\Deps$ was defined in \eqref{eq:diagonal} it holds that $(\clDeps^c)_x=(\clball{\eps}{x})^c$ which implies that
    \begin{align}
        p_n(\clDeps^c)=\int_\X m_{n,x}((\clball{\eps}{x})^c)\dd\vrho(x)=0
    \end{align}
    and thus for all $f\in C_b(\X\times\X)$ that are only supported outside of $\clDeps$ we have
    \begin{align}
        \int_\X\int_\X f(x,y)\dd p(x,y)=\lim_{k\to\infty} \int_\X\int_\X f(x,y)\dd p_{n_k}(x,y)=0
    \end{align}
    and by \cref{lem:measure support} we obtain $\supp p\subset \clDeps$. 
    Similarly as before we show that
    \begin{align}
        p(E)=\int_\X m_x(E_x)\dd\vrho(X)\qquad\text{for }E\in \BB_\X\otimes\BB_\X.
    \end{align}
    Hence, 
    \begin{align}
        \int_\X m_x( (\clball{\eps}{x})^c)\dd\vrho(x)=p(\clDeps^c)=0
    \end{align}
    implying that $\supp m_x\subset\clball{\eps}{x}$ for $\vrho$-a.e. $x\in\X$.
\end{proof}
Now that we secured the compactness of ${\MM_\eps}$ by means of \cref{lem:compactness_RW} we can finally prove \cref{thm:u cont-subdiff characterization}.
\begin{proof}[Proof of \cref{thm:u cont-subdiff characterization}]
    By \cref{thm:u cont-dual representation} we have for $u\in C_0(\X)$ that
    \begin{align}
        \tv_\eps(u)&=
        \max_{\Boldserifm\in{\MM_\eps}\times{\MM_\eps}}-\int_\X u(x)\dd\div_\eps[\Boldserifm;\Frho](x)
        \\
        &=
        \max_{\mu\in\M_f^\pm(\X)}\int_\X u(y) \dd\mu(y)-\chi_{P}(\mu)=\chi_P^*(u)
    \end{align}
    where $P\coloneq \set{-\div_\eps[\Boldserifm;\Frho]}{\Boldserifm\in{\MM_\eps}\times{\MM_\eps}}\subset M^\pm_f(\X)$. 
    As shown in \cite[Example 4.3]{ekeland1999}, for $\mu\in \M_f(\X)$ it holds that
    \begin{align}
        \tv^*_\eps(\mu)=\chi^{**}_P(\mu)=\chi_{\overline{\mathrm{conv}}(P)}(\mu),
    \end{align}
    where the closure is taken with respect to the weak-* topology on $\M_f^\pm(\X)$ which, thanks to $\X$ being compact, coincides with the weak topology $\tau_w$ from \cref{defi:weak measure convergence}, cf. \cref{rem:weak_topology}. Applying the Fenchel--Young inequality, we have that $\mu^*\in\p\tv_\eps(u)$ if and only if 
    \begin{align}
       \tv_\eps(u)+\tv_\eps^*(\mu^*)=\int_\X u\dd \mu^*.
    \end{align}
    Hence, the proof is finished if we show that $P$ is convex and closed. First, to prove convexity note that for $\serifm^0,\serifm^1\in{\MM_\eps}$  their convex combination $\serifm^\theta\in RW(\X)$ which is defined through $m_x^\theta \coloneq (1-\theta)m_x^0 + \theta m_x^1$ for $\theta\in[0,1]$ satisfies 
    \begin{align}
        \supp m^\theta_x=\supp((1-\theta) m_{x}^0)\cup\supp (\theta m^1_x)=\supp m^0_x\cup\supp m^1_x\subset{\clball{\eps}{x}}
    \end{align}
    for $\vrho$-a.e. $x\in\X$ which implies that ${\MM_\eps}$ is convex. 
    Due to the linearity of $\div_\eps$, the set $P$ is a convex set. 
    Next, let $(p_n)_{n\in\IN}\subset P$ and $p$ be a finite signed measure such that $p=\wlim{n\to\infty}p_n$ is the limit with respect to weak measure convergence. Then, 
    by definition for each $p_n$ there exists $\Boldserifm_n\in{\MM_\eps}\times{\MM_\eps}$ such that $p_n=\div_\eps[\Boldserifm_n;\Frho]$. 
    We apply \cref{lem:compactness_RW} to both components $\serifm_n^i$ for $i\in\{0,1\}$ and $\grad_\eps[f]\in C(\X\times\X)$ for $f\in C(\X)$ to obtain that for a subsequence (which we do not relabel) we have
    \begin{align}
        \lim_{n\to\infty} -\int_\X\int_\X \grad_\eps[f](x,y)\dd m_{n,x}^i(y)\dd\vrho_i(x)=-\int_\X\int_\X \grad_\eps[f](x,y)\dd m_{x}^i(y)\dd\vrho_i(x)
    \end{align}
    and $\supp m^i_x\subset\clball{\eps}{x}$ for $\vrho$-a.e. $x\in\X$ and $i\in \{0,1\}$. 
    We also have that
    \begin{align}
        \int_\X\int_\X \grad_\eps[f](x,y)\dd m^i_{n,x}(y)\dd\vrho_i(x)=\int_{\X\times\X}\grad_\eps[f](x,y)\dd(\vrho_i\ltimes\serifm_n^i)(x,y)
    \end{align}
    for all $n\in \IN$ and the same is true for $\serifm^i$.
    Hence, we get
    \begin{align}
        \int_\X f\dd p
        &=
        \lim_{n\to\infty}
        \int_\X f \dd p_n
        =\lim_{n\to\infty} \int_\X f \dd\div_\eps[\Boldserifm_n;\Frho]
        =\lim_{n\to\infty} -\int_{\X\times\X}\grad_\eps[f]\dd [\Boldserifm_n;\Frho]\\
        &=-\int_{\X\times\X}\grad_\eps[f]\dd[\Boldserifm;\Frho]
        =
        -\int_\X
        f(x)\dd \div_\eps[\Boldserifm;\Frho](x).
    \end{align}
    Since $\X$ is compact, we have $C(\X)=C_0(\X)$ and \cref{prop:u cont-nonlocal divergence solution} shows that $p=-\div_\eps[\Boldserifm;\Frho]\in P$ as equality of finite signed Radon measures.
\end{proof}
To conclude this section, we now show that our result matches the subgradient formula derived in \cite{bungert2024mean} in the setting where $\X=\Omega\subset\IRN$ is a bounded domain and $u\in C^2(\overline{\Omega})$ is such that $\abs{\nabla u}\geq c$ in $\overline{\Ome}$ for a constant $c>0$. 
In this case, an explicit subgradient of $u$ was constructed in~\cite{bungert2024mean}, but it was not investigated whether this is the only subgradient.
For $\Ome$ define the inner parallel set as $\Ome_\eps\coloneq \set{x\in\Ome}{\mathrm{dist}(x,\IRN\setminus \Ome)>\eps}$. 
Then, under the assumptions above for any $x\in\Ome_\eps$ the sets $\argmax_{{\clball{\eps}{x}}}u$ and $\argmin_{{\clball{\eps}{x}}}u$ are singletons and induce $C^1$ diffeomorphisms which was used in \cite{bungert2024mean} to construct the subgradient.
Here, we follow are more general approach in our metric measure space setting and show that if the $\argmax$ and $\argmin$ sets are singletons, so is the subdifferential and the unique subgradient can be explicitly written down.
\begin{example}\label{ex:u cont-subdiff for singletons}
    Under the conditions of \cref{thm:u cont-subdiff characterization} and for any $u\in C(\X)$ such that the maps
    \begin{align}
        \Gamma_\eps(x)\coloneq \argmax_{{\clball{\eps}{x}}}u\qquad\text{and}\qquad\gamma_\eps(x)\coloneq \argmin_{{\clball{\eps}{x}}}u
    \end{align}
    are singletons for every $x\in\X$, the subdifferential is given by $\p\tv_\eps(u)=\{p\}$ where $p$ is the finite signed Radon measure defined by
    \begin{align}
        p(A)\coloneq \frac{(\Gamma_\eps)_\sharp \vrho_0(A)-\vrho_0(A)}{\eps}+\frac{\vrho_1(A)-(\gamma_\eps)_\sharp \vrho_1(A)}{\eps},
        \qquad A\in\BB_\X.
    \end{align}
    First, due to the assumption on $u$ the mappings $\Gamma_\eps$ and $\gamma_\eps$ reduce to measurable functions on $\X$, see \cref{rmk:Gamma_eps measurable}, and we can rewrite		
        \begin{align}
			\tv_\eps(u)
            &=
            \frac{1}{\eps}\left[\int_\X \big(u(\Gamma_\eps(x))-u(x)\big)\dd\vrho_0(x)+\int_\X\big(u(x)-u(\gamma_\eps(x))\big)\dd \vrho_1(x)\right].
		\end{align}
		On the other hand, due to \cref{thm:u cont-subdiff characterization} for any $\mu^*\in\p\tv_\eps(u)$ there exists $\Boldserifm\in {\MM_\eps}\times{\MM_\eps}$ such that
		\begin{align}
			\tv_\eps(u)
			&=
            \int_\X u\dd \mu^*
			\\
            &=
            \frac{1}{\eps}
            \int_\X
            \left(\int_\X u(y) \dd m_x^0(y)
            - u(x)\right) \dd\vrho_0(x)
            +\frac{1}{\eps}
            \int_\X
            \left(u(x)-\int_\X u(y) \dd m_x^1(y)\right)
            \dd\vrho_1(x).
		\end{align}
		Subtracting both and using that $u(\Gamma_\eps(x))\geq u(y)$ (and similarly $u(\gamma_\eps(x))\leq u(y)$) for all $y\in\clball{\eps}{x}$ yields
		\begin{align}
			\int_\X u(y) \dd m_x^0(y)= u(\Gamma_\eps(x))\quad\text{and}\quad \int_\X u(y)\dd m_x^1(y)=u(\gamma_\eps(x))
		\end{align}
		for $\vrho_0$-a.e. and $\vrho_1$-a.e. $x\in\X$.
		In fact, we even get $u(y)=u(\Gamma_\eps(x))$ for $m_x^0\mhyphen$a.e. $y\in {\clball{\eps}{x}}$ (and similarly for $\gamma_\eps$). Finally, $\Gamma_\eps(x)$ being the unique maximizer of $u$ over ${\clball{\eps}{x}}$ we obtain
		\begin{align}
			m_x^0(\set{y\in{\clball{\eps}{x}}}{u(y)<u(\Gamma_\eps(x))})=m_x^0({\clball{\eps}{x}}\setminus\Gamma_\eps(x))=0
		\end{align}
		and hence $m_x^0=\delta_{\Gamma_\eps(x)}$ for $\vrho_0\mhyphen$a.e. $x\in\X$. Similarly, $m_x^1=\delta_{\gamma_\eps(x)}$ for $\vrho_1\mhyphen$a.e. $x\in\X$. Note, that by definition of the pushforward measure one can rewrite
        \begin{align}
            (\Gamma_\eps)_\sharp\vrho_0(A)=\vrho_0\left(\Gamma_\eps^{-1}(A)\right)=\int_\X\Ie_A\left(\Gamma_\eps(x)\right)\dd\vrho_0(x)=\int_\X \delta_{\Gamma_\eps(x)}(A)\dd\vrho_0(x)
        \end{align}
        for any $A\in\BB_\X$.
        So, there is only one subgradient $\mu^*$ of $\tv_\eps(u)$ which is given by
		\begin{align}
			\mu^*(A)		&=\frac{\int_\X\delta_{\Gamma_\eps(x)}(A)\dd\vrho_0(x)-\vrho_0(A)}{\eps}+\frac{\vrho_1(A)-\int_\X\delta_{\gamma_\eps(x)}(A)\dd\vrho_1(x)}{\eps}\\
			&=\frac{(\Gamma_\eps)_\sharp \vrho_0(A)-\vrho_0(A)}{\eps}+\frac{\vrho_1(A)-(\gamma_\eps)_\sharp \vrho_1(A)}{\eps}=p(A),
            \qquad A\in\BB_\X.
		\end{align}
\end{example}
For the sake of completeness we prove that in the setting of \cref{ex:u cont-subdiff for singletons} the mappings $\Gamma_\eps$ and $\gamma_\eps$ indeed reduce to measurable functions on $\X$.
\begin{remark}[Measurability of $\Gamma_\eps$ and $\gamma_\eps$]\label{rmk:Gamma_eps measurable}
    We only show that $\Gamma_\eps$ is a measurable function since $\gamma_\eps$ follows analogously. Knowing that $\Gamma_\eps(x)$ is a singleton for every $x\in\X$, it is clear that the correspondence $\Gamma_\eps$ reduces to a function. To prove its measurability, define for any $x\in\X$ the argmax set, in the spirit of \cref{prop:u cont-measurable selector}, by
    \begin{align}
        A^x\coloneq \set{m\in\P(\clball{\eps}{x})}{\int_{\clball{\eps}{x}}u\dd m=u^\eps(x)}
    \end{align}
    where $u^\eps$ is the duality formulation of the maximum as defined in \eqref{eq:u cont-max problem}. By \cref{prop:u cont max dualisation} we have that $u^\eps(x)=\max_{\clball{\eps}{x}}u$ and thus
    \begin{align}
        A^x=\set{m\in\P(\clball{\eps}{x})}{\int_{\clball{\eps}{x}}u\dd m=u(\Gamma_\eps(x))}.
    \end{align}
    Using the same arguments as in \cref{ex:u cont-subdiff for singletons} we obtain that $A^x=\{\delta_{\Gamma_\eps(x)}\}$ is a singleton and by \cref{prop:u cont-measurable selector} there exists a measurable function $D\colon \X\to\P(\X)$ given by $D(x)=\delta_{\Gamma_\eps(x)}$. Finally, this implies that also $\Gamma_\eps$ is measurable.
    To see this, note that for any continuous function $u\in C(\X)$ we have
    \begin{align}
        u(\Gamma_\eps(x)) = \int_\X u(y)\d\delta_{\Gamma_\eps(x)}(y).
    \end{align}
    Since $\X$ is compact, the right hand side is the composition of the measurable function $D$ with a continuous function, namely the integral against $u$. 
    Hence, we get that for any $u\in C(\X)$ the map $x\mapsto u(\Gamma_\eps(x))$ is measurable.
    It remains to show that this implies measurability of $\Gamma_\eps$.
    Let $A\in\BB_\X$ be an open set and choose continuous function $u_n$ such that $u_n(x)\to 1_A(x)$ for all $x\in \X$.
    Applying this shows $u_n(\Gamma_\eps(x))\to 1_A(\Gamma_\eps(x))$ and, since the pointwise limit of measurable functions is measurable, the function $x\mapsto 1_A(\Gamma_\eps(x))$ is measurable.
    This is equivalent to $\Gamma_\eps^{-1}(A)$ being measurable and hence $\Gamma_\eps$ is measurable.    
\end{remark}

\section{Dualization for \texorpdfstring{$L^\infty(\X)$}{L∞(X)}}\label{sec:bounded}

In the previous section we showed that the adversarial total variation functional defined on the space of continuous functions on $\X$ admits a duality formulation and we characterized the subdifferential for compact $\X$. 
In this section, we pursue the same goal while dropping the assumption that $u$ is continuous. 
This naturally leads to a smaller set of admissible test functions and hence to a different characterization of the subdifferential.
Furthermore, it reintroduces the reference measure $\nu$ in \eqref{eq:TV functional} which was irrelevant in the previous section.

In \cref{sec:dualization_bounded} we derive duality formulations for the two parts of the nonlocal total variation functional defined on $L^\infty(\X)$ as a base space. With $L^p(\X)$ for $p\in[1,\infty]$ we denote the standard Lebesgue spaces on $\X$ with respect to the measure $\nu$. 
Throughout, $\X$ is a separable metric space and $\nu$ is the reference measure which we assume to be $\sigma$-finite on $\X$ and that satisfies $\vrho_0,\vrho_1\ll \nu$. As in \cref{sec:continuous} we define $\vrho\coloneq \vrho_0+\vrho_1$ and for a fixed $\eps>0$ we remind the reader of the definition of the open $\eps$-strip around the diagonal of $\X\times\X$ from \eqref{eq:diagonal}: 
\begin{align}
\Deps\coloneq\set{(x,y)\in\X\times\X}{\dd(x,y)<\eps}.
\end{align}
Next, in \cref{sec:grad_div_bounded} we use an integration-by-parts identity involving suitable notions of a nonlocal gradient and corresponding nonlocal divergence which leads to final dualization of $\nu\mhyphen\tv_\eps$. 
We also show the consistency of these operators with their local counterparts in the Euclidean setting in a pointwise and a distributional sense.
In \cref{sec:subdiff_bounded} we conclude with a limit characterization of the subdifferential of the adversarial total variation as subset of $L^\infty(\X)^*$.

\subsection{Dual Reformulation for Essentially Bounded Functions}\label{sec:dualization_bounded}

Similar to \cref{sec:dualization_continuous} we first define a suitable set of test functions to dualize the individual parts of the total variation.
Instead of random walks we shall use the following smaller set $\PP_\eps\subset\MM_\eps$ of $L^1$-functions which, however, is not weak-* closed in $L^\infty(\X)^*$. 
This will not cause any problems for the dualizations of \cref{sec:dualization_bounded,sec:grad_div_bounded} but will make the characterization of subgradients in \cref{sec:subdiff_bounded} harder.
The test set is defined as follows:
\begin{align}\label{eq:u bounded-test function set}
    {\PP_\eps}\coloneq  \left\{ 
        \Psi\in L^1(\X\times\X;\vrho\otimes\nu) \, \middle| \, 
        \begin{gathered}
            \Psi\geq 0\ \vrho\otimes\nu\text{-a.e. on }\X\times\X,\quad\Psi=0\ \vrho\otimes\nu\text{-a.e. on }\Deps^c,\\ 
            \int_\X\Psi(x,y)\mathrm{d} \nu(y)=1\ \text{for }\vrho\text{-a.e. }x\in \X
        \end{gathered}
        \right\}.
\end{align}
In this whole section all $L^1$ and $L^\infty$ Lebesgue spaces will be with respect to the measure $\nu$ or $\nu\otimes\nu$, respectively. 
The only exception is the definition of ${\PP_\eps}$ above where indeed the appropriate measure is $\rho\otimes\nu$.
We now state the duality reformulation of both parts of the nonlocal adversarial total variation functional which we will use later on in \cref{thm:u bounded-dual representation} to dualize the total variation involving a nonlocal gradient and divergence.

\begin{theorem}[Dual reformulation]\label{thm:dual representation for bounded functions}
    Let $\X$ be a separable metric space, $\mu\in\P(\X\times\{0,1\})$ a probability measure, and assume that the respective conditional distributions $\vrho_i\coloneq\mu(\filll\times\{i\})$ for $i\in\{0,1\}$ are absolutely continuous with respect to a $\sigma$-finite reference measure $\nu$. 
    Then, for $u\in L^\infty(\X)$ we have 
    \begin{align}
        \int_\X \nesssup{u}{\ball{\eps}{x}}\dd\vrho_0(x)=\sup_{\Psi\in{\PP_\eps}} \int_\X\int_\X \Psi(x,y)u(y)\dd \nu(y)\dd \vrho_0(x)
    \end{align}
    and
    \begin{align}
        \int_\X \nessinf{u}{\ball{\eps}{x}}\dd\vrho_1(x)=\inf_{\Psi\in{\PP_\eps}} \int_\X\int_\X \Psi(x,y)u(y)\dd \nu(y)\dd \vrho_1(x).
    \end{align}
\end{theorem}
\begin{remark}[Adversarial total variation for unbounded functions]
    For generic non-zero data distribution measures $\vrho_0$ and $\vrho_1$ one can easily find $u\in L^1(\X)\setminus L^\infty(\X)$ such that $\nu\mhyphen\tv_\eps(u)=\infty$.
    Under certain technical assumptions on the supports of the measures (e.g., if they are not both fully supported), one can identify cases where $\nu\mhyphen\tv_\eps(u)<\infty$ despite $u\notin L^\infty(\X)$. 
    To keep the exposition simple, we therefore restrict ourselves in the following to essentially bounded functions, which covers most relevant situations from a practical point of view.
\end{remark}
Before proving the theorem we show that the essential supremum and infimum are both measurable functions with respect to the Borel $\sigma$-algebra $\BB_\X$. We set $\overline{\IR}\coloneq\IR\cup\{-\infty,+\infty\}.$
\begin{lemma}\label{lem:u bounded-esssup measurable}
    Let $\X$ be a separable metric space with $\sigma$-finite measure $\nu$. For a measurable function $u\colon \X\to\IR$ define $\mathfrak{s},\mathfrak{i}\colon\X\to\overline{\IR}$ by
    \begin{align}
        \mathfrak{s}(x)\coloneq \nesssup{u}{\ball{\eps}{x}}\qquad\text{and}\qquad\mathfrak{i}(x)\coloneq \nessinf{u}{\ball{\eps}{x}}
    \end{align}
    for $x\in \X$. Then, both $\mathfrak{s}$ and $\mathfrak{i}$ are $\BB_\X$-measurable functions.
\end{lemma}
\begin{proof}
    First, fix $c\in\IR$ and define the superlevel set
    \begin{align}
    M_c\coloneq \set{x\in\X}{u(x)>c}=u^{-1}((c,\infty)),
    \end{align}
    which is measurable since $u$ is measurable. The distance function $\dd$ being continuous, we obtain that $\Deps=\dd^{-1}([0,\eps))\in\BB_{\X\times\X}$. Due to the separability of $\X$ this implies that $\Deps\in\BB_\X\otimes\BB_\X$, see \cite[Theorem 4.44]{aliprantis2006}, and thus the map $(x,y)\mapsto=\Ie_{\Deps}(x,y)=\Ie_{\ball{\eps}{x}}(y)$ is jointly measurable. Hence, applying the Fubini--Tonelli theorem yields that the function
    \begin{align}
        h_c(x)\coloneq \int_\X \Ie_{M_c}(y)\Ie_{\ball{\eps}{x}}(y)\dd\nu(y)=\nu(M_c\cap\ball{\eps}{x})
    \end{align}
    is measurable in $x\in\X$. Furthermore, we have
    \begin{align}
        \mathfrak{s}^{-1}((c,\infty])&=\set{x\in\X}{\nesssup{u}{\ball{\eps}{x}}>c}=\set{x\in\X}{\nu\big(\set{y\in\ball{\eps}{x}}{u(y)>c}\big)>0}\\&=\set{x\in\X}{\nu\big(M_c\cap\ball{\eps}{x}\big)>0}=\set{x\in\X}{h_c(x)>0}=h_c^{-1}((0,\infty]).
    \end{align}
    So the measurability of $h_c$ implies that $\mathfrak{s}^{-1}((c,\infty])\in\BB_\X$ for every $c\in\IR$ and because these intervals generate $\BB(\overline{\IR})$, $\mathfrak{s}$ is $\BB_\X$-measurable. By the fact that $-u$ is measurable for any measurable function $u$ and the observation that $\mathfrak{i}(x)=-\nesssup{-u}{\ball{\eps}{x}}$, the measurability of $\mathfrak{i}$ follows directly.
\end{proof}
Having established measurability of the essential supremum and infimum we are now in position to prove the main result. We first discuss a preliminary observation before we start the actual proof.
\begin{remark}\label{rem:u bounded-balls have positive measure}
    In the situation of \cref{thm:dual representation for bounded functions}, balls around $\vrho$-a.e. point in $\X$ have positive $\nu$-measure. To see this, define $Z\coloneq \set{x\in\X}{\nu(\ball{\eps}{x})=0}$ which is measurable due to the measurability of $D$ which we showed in the proof of \cref{lem:u bounded-esssup measurable}. 
    If $Z=\emptyset$ there is nothing to show.
    Hence, suppose $Z\neq\emptyset$ and note that $\set{\ball{\eps}{q}}{q\in Z}$ is an open cover of $Z$. Then, since $\X$ is a separable, so is $Z$ with the induced metric. In particular, $Z$ is a Lindelöf space and hence there exists a countable subcover such that $\bigcup_{i\in\IN} \ball{\eps}{q_i}\supset Z$. Therefore, $\nu(Z)\leq\nu(\bigcup_{i\in\IN} \ball{\eps}{q_i})\leq \sum_{i\in\IN} \nu(\ball{\eps}{q_i})=0$ since every $q_i\in Z$ and due to the absolute continuity of $\vrho$ with respect to $\nu$ we infer that $\vrho(Z)=0$. Consequently, $\nu(\ball{\eps}{x})>0$ for $\vrho$-a.e. $x\in\X$.
\end{remark}
We now prove \cref{thm:dual representation for bounded functions} by means of proving a two-sided inequality. 
The core of the proof is that we can identify a measurable set whose indicator function, suitably normalized, yields a sharp lower bound for the supremum over ${\PP_\eps}$.
\begin{proof}[Proof of \cref{thm:dual representation for bounded functions}]
    We only consider the first identity to be proved in \cref{thm:dual representation for bounded functions} and start with the simple inequality.
    Fix an arbitrary $\Psi\in{\PP_\eps}$ and note that due to \cref{lem:u bounded-esssup measurable} the essential supremum is measurable. Due to \cref{rem:u bounded-balls have positive measure} and $u$ being essentially bounded on $\X$ we infer that $\nesssup{u}{\ball{\eps}{x}}$ is finite for $\vrho$-a.e. $x\in \X$ and all of the following integrals are well-defined. Since $\Psi\in {\PP_\eps}$, we have
    \begin{align}
        &\int_\X\int_\X \Psi(x,y)u(y)\dd\nu(y)\dd\vrho_0(x)=\int_\X\int_{\ball{\eps}{x}}\Psi(x,y) u(y)\dd\nu(y)\dd\vrho_0(x)\\
        &\qquad\leq \int_\X\int_{\ball{\eps}{x}}\Psi(x,y) \nesssup{u}{\ball{\eps}{x}} \dd\nu(y)\dd\vrho_0(x)=\int_\X\nesssup{u}{\ball{\eps}{x}}\dd\vrho_0(x).
    \end{align}
    Since $\Psi\in {\PP_\eps}$ was arbitrarily chosen, we obtain
    \begin{align}
        \sup_{\Psi\in{\PP_\eps}}\int_\X\int_\X \Psi(x,y) u(y) \dd\nu(y)\dd\vrho_0(x)\leq \int_\X \nesssup{u}{\ball{\eps}{x}}\dd\vrho_0(x).
    \end{align}
    For the converse inequality, we can without loss of generality assume that $u$ is Borel measurable.
    Fix $\delta>0$ and define
    \begin{align}
        A_\delta\coloneq \set{(x,y)\in\X\times\X}{y\in\ball{\eps}{x}\quad\text{and}\quad u(y)\geq \nesssup{u}{\ball{\eps}{x}}-\delta}=D\cap F^{-1}([0,\infty))
    \end{align}
    where $F\colon \X\times\X\to\IR$ is defined by
    \begin{align}
        F(x,y)=u(\pi_2(x,y))-\nesssup{u}{\ball{\eps}{\pi_1(x,y)}}+\delta\qquad\text{for }x,y\in\X.
    \end{align}
    Here, $\pi_1,\pi_2$ denote the projections onto the first and second component and $F$ is well-defined since $u$ is $\nu$-a.e. finite and the essential supremum is $\vrho$-a.e. finite as argued above. As $\pi_1,\pi_2$ are both $(\BB_\X\otimes\BB_\X,\BB_\X)$-measurable and $u$ and $\nesssup{u}{\ball{\eps}{\cdot}}$ are $\BB_\X$-measurable, see \cref{lem:u bounded-esssup measurable}, we infer that $F$ is a jointly measurable function. Consequently, $A_\delta$ is measurable as an intersection of measurable sets.
    Let us now construct the maximizing test function. Since $\nu$ is $\sigma$-finite, there exist pairwise disjoint measurable sets $S_i$ for $i\in\IN$ satisfying $\nu(S_i)<\infty$ and $\bigcup_{i\in\IN}S_i=\X$. Define
    \begin{align}
        g\coloneq \sum_{i\in\IN}\frac{2^{-i}}{1+\nu(S_i)}\Ie_{S_i}
    \end{align}
    which is a measurable and strictly positive function by construction. Furthermore,
    \begin{align}
        \int_\X g\dd\nu=\sum_{i\in\IN}\frac{2^{-i}}{1+\nu(S_i)}\nu(S_i)\leq \sum_{i\in\IN} 2^{-i}=1.
    \end{align}
    In particular, $\int_E g\dd\nu>0$ if and only if $\nu(E)>0$. Moreover, because each section of $A_\delta$, given by $A_{\delta,x}\coloneq\set{y\in\X}{(x,y)\in A_\delta}$, has positive $\nu$-measure for $\vrho$-a.e. $x\in \X$, we obtain that $\int_\X\Ie_{A_\delta}(x,y) g(y)\dd\nu(y)=\int_{A_{\delta,x}} g(y)\dd\nu(y)>0$ for $\vrho$-a.e. $x\in \X$. Hence, for $\vrho$-a.e. $x\in\X$ we have
    \begin{align}
        0<m(x)\coloneq \int_\X \Ie_{A_\delta}(x,y)g(y)\dd\nu(y)\leq 1,
    \end{align}
    and $m$ is measurable by Tonelli's theorem.
    We define
    \begin{align}
        \hat{\Psi}(x,y)\coloneq \begin{cases} \frac{g(y)}{m(x)}\Ie_{A_\delta}(x,y)&\text{if }m(x)>0,\\0& \text{otherwise,}\end{cases}
    \end{align}
    which is well-defined and measurable by the preceding discussion. Moreover, by construction $\int_\X \hat{\Psi}(x,y)\dd\nu(y)=1$ for $\vrho$-a.e. $x\in\X$ and $\hat{\Psi}=0$ $\vrho\otimes\nu$-a.e. on $D^c$. In summary, $\hat{\Psi}\in{\PP_\eps}$ and for $\vrho$-a.e. $x\in\X$ it holds that
    \begin{align}
        \int_\X\hat{\Psi}(x,y)u(y)\dd\nu(y)=\frac{1}{m(x)}\int_{A_{\delta,x}}u g\dd\nu\geq \frac{1}{m(x)} \int_{A_{\delta,x}}(\nesssup{u}{\ball{\eps}{x}}-\delta)g\dd\nu=\nesssup{u}{\ball{\eps}{x}}-\delta.
    \end{align}
    Integrating against $\vrho_0$ yields
    \begin{align}
        \sup_{\Psi\in{\PP_\eps}}\int_\X\int_\X \Psi(x,y) u(y) \dd\nu(y)\dd\vrho_0(x)\geq \int_\X\nesssup{u}{\ball{\eps}{x}}\dd\vrho_0(x)-\delta \vrho_0(\X).
    \end{align}
    By definition, the measure $\vrho_0$ is finite and thus sending $\delta$ to zero completes the proof of the first statement. The second follows by an analogous argument.        
\end{proof}

\subsection{Nonlocal Gradient and Divergence}\label{sec:grad_div_bounded}
With the dual representation of \cref{thm:dual representation for bounded functions} at hand, we rearrange the terms to obtain reformulations involving the same nonlocal gradient as in the continuous case, now defined as linear operator from $L^\infty(\X)$ to $L^\infty(\X\times\X)$, and a corresponding nonlocal divergence. A similar definition of nonlocal operators can be found in \cite{gilboa2009nonlocal}.
\begin{definition}[Nonlocal gradient]\label{def:nonlocal_gradient_bounded}
    For $\eps>0$ the nonlocal gradient $\grad_\eps\colon L^\infty(\X)\to L^\infty(\X\times\X)$ is defined by
    \begin{align}
        \grad_\eps[u](x,y) \coloneq \frac{u(y)-u(x)}{\eps},\qquad x,y\in\X,
    \end{align}
    for $u\in L^\infty(\X)$.
\end{definition}
Similar to \cref{def:nonlocal_divergence_random_walk} we define a nonlocal divergence that is the preadjoint operator of the gradient. 
More precisely, since we do not define the divergence on the dual space of $L^\infty(\X\times\X)$ but rather on its predual $L^1(\X\times\X)$, the gradient operator is the adjoint of the divergence and not the other way round like in \cref{sec:continuous}.
\begin{definition}[Nonlocal divergence]\label{def:u bounded-nonlocal divergence}
    For $\eps>0$ the nonlocal divergence $\div_\eps\colon L^1(\X\times\X)\to L^1(\X)$ is defined by
    \begin{align}
        \div_\eps[f](y)\coloneq \int_{\X}\frac{f(y,x)-f(x,y)}{\eps}\dd \nu(x)\qquad\text{for all }y\in \X,
    \end{align}
    for $f\in L^1(\X\times\X)$.
\end{definition}
\begin{remark}
    Note that when inserting $\Psi\in{\PP_\eps}$ we implicitly only integrate over $\ball{\eps}{y}$ for each $y\in\X$ due to the support condition. 
    For consistency with the definitions from \cref{sec:bounded} we, however, do not include the integration domain $\ball{\eps}{x}$ into the definition of the nonlocal divergence or its adjoint.
\end{remark}
Those definitions lead to a similar duality reformulation of the total variation as the one we obtained for the continuous case.
In the rest of this section we abbreviate
\begin{align}
    \tv_\eps(u) := 
    \int_\X \frac{\nesssup{u}{\ball{\eps}{x}}-u(x)}{\eps}\dd\vrho_0(x)+\int_\X \frac{u(x)-\nessinf{u}{\ball{\eps}{x}}}{\eps}\dd\vrho_1(x)
\end{align}
where $u\in L^\infty(\X)$ and the essential supremum and infimum is taken with respect to the reference measure $\nu$.
For notational convenience, we denote the vector of densities by $\Frho= (\vrho_0,\vrho_1)$ and the vector of test functions by $\FPsi=(\Psi^0,\Psi^1)\in {\PP_\eps}\times{\PP_\eps}$ and define the antisymmetric pairing $[\FPsi;\Frho]\coloneq (\Psi^0\vrho_0)-(\Psi^1\vrho_1)$ as a function from $\X\times\X$ to $\IR$, where $(\Psi^i\vrho_i)(x,y)\coloneq \Psi^i(x,y)\frac{\dd\vrho_i}{\dd\nu}(x)$ for $x,y\in \X$.
Note that in the following, we shall write $\div_\eps\left[\FPsi;\Frho\right]$ in place of $\div_\eps\left[[\FPsi;\Frho]\right]$.
\begin{remark}[Overload of notation]
    To emphasize the analogy between the two main sections of this paper, we use the same notation for the respective key definitions of $\grad_\eps$, $\div_\eps$, and the antisymmetric pairing $[\,\cdot\,;\,\cdot\,]$.
\end{remark}
\begin{theorem}[Dual representation of $\tv_\eps$]\label{thm:u bounded-dual representation}
    Let $\X$ be a separable metric space, $\mu\in\P(\X\times\{0,1\})$ a probability measure, and assume that the respective conditional distributions $\vrho_i\coloneq\mu(\filll\times\{i\})$ for $i\in\{0,1\}$ are absolutely continuous with respect to a $\sigma$-finite reference measure $\nu$. Then, for $u\in L^\infty(\X)$ the adversarial total variation $\tv_\eps$ admits the following representations
    \begin{align}
        \tv_\eps(u)&=\sup_{\FPsi\in{\PP_\eps}\times{\PP_\eps}} \int_{\X\times\X}\grad_\eps[u]\cdot[\FPsi;\Frho]\dd (\nu\otimes\nu)=\sup_{\FPsi\in{\PP_\eps}\times{\PP_\eps}}-\int_{\X} u \cdot\div_\eps\left[\FPsi;\Frho\right]\dd \nu.
    \end{align}
\end{theorem}
\begin{proof}
    The first equality is shown in \cref{prop:u bounded-nonlocal gradient} and the second one in \cref{cor:u bounded-nonlocal divergence}, both below.
\end{proof}
\begin{remark}
    Note that for $\FPsi\in{\PP_\eps}\times{\PP_\eps}\subset L^1(\X\times\X;\rho\otimes\nu)^2$ the function $\left[\FPsi;\Frho\right]$ indeed lies in $L^1(\X\times\X;\nu\otimes\nu)$ such that the right hand side is well-defined, taking into account that the divergence $\div_\eps$ from \cref{def:u bounded-nonlocal divergence} was defined on functions in $L^1(\X\times\X;\nu\otimes\nu)$.
\end{remark}

We now prove the two identities in \cref{thm:u bounded-dual representation}, involving the nonlocal gradient and divergence, respectively.
Applying \cref{thm:dual representation for bounded functions} and a straightforward rearrangement similar to \cref{prop:u cont-nonlocal gradient} leads to the first equality in \cref{thm:u bounded-dual representation}.
\begin{proposition}\label{prop:u bounded-nonlocal gradient}
    Under the conditions of \cref{thm:u bounded-dual representation} for all $u\in L^\infty(\X)$ we have 
    \begin{align}
        \tv_\eps(u)
        =\sup_{\FPsi\in{\PP_\eps}\times{\PP_\eps}} \int_{\X\times\X} \grad_\eps[u]\cdot[\FPsi;\Frho]\dd(\nu\otimes\nu) .
    \end{align}
\end{proposition}
\begin{proof}
    Due to \cref{thm:dual representation for bounded functions} and Fubini's theorem we have
    \begin{align}
        \frac{1}{\eps}\int_{\X} [\nesssup{u}{\ball{\eps}{x}}&-u(x)] \dd\vrho_0(x)=\frac{1}{\eps}\left[\sup_{\Psi\in{\PP_\eps}} \int_{\X}\int_{\X} \Psi(x,y) u(y) \dd \nu(y)\dd \vrho_0(x)-\int_{\X} u(x)\dd\vrho_0(x)\right]\\
        &=\sup_{\Psi\in {\PP_\eps}} \frac{1}{\eps}\left[\int_{\X}\int_{\X}\Psi(x,y) u(y)\dd\nu(y)-\int_{\X}\Psi(x,y)\dd \nu(y)\, u(x)\dd\vrho_0(x)\right]\\
        &=\sup_{\Psi\in{\PP_\eps}}\int_{\X}\int_{\X} \Psi(x,y)\frac{u(y)-u(x)}{\eps}\dd \nu(y)\dd \vrho_0(x)\\
        &=\sup_{\Psi\in{\PP_\eps}}\int_{\X\times\X} \Psi(x,y)\frac{\dd\vrho_0}{\dd\nu}(x)\grad_\eps[u](x,y) \dd (\nu\otimes\nu)
    \end{align}
    and similarly
    \begin{align}
        \frac{1}{\eps}\int_{\X} [u(x)&-\nessinf{u}{\ball{\eps}{x}}]\dd\vrho_1(x)=\frac{1}{\eps}\left[\int_{\X} u(x)\dd\vrho_1(x)-\inf_{\Psi\in {\PP_\eps}}\int_{\X}\int_{\X} \Psi(x,y)u(y)\dd \nu(y)\dd \vrho_1(x)\right]\\
        &=\frac{1}{\eps}\left[\int_{\X} u(x)\dd\vrho_1(x)+\sup_{\Psi\in {\PP_\eps}}\int_{\X}\int_{\X} \Psi(x,y)(-u(y))\dd \nu(y)\dd \vrho_1(x)\right]\\
        &=\sup_{\Psi\in{\PP_\eps}}\frac{1}{\eps}\int_{\X}\int_{\X} \Psi(x,y)\dd \nu(y)\,u(x)-\int_{\X}\Psi(x,y) u(y)\dd \nu(y)\dd \vrho_1(x)\\
        &=\sup_{\Psi\in{\PP_\eps}} \int_{\X}\int_{\X} \Psi(x,y) \frac{u(x)-u(y)}{\eps}\dd \nu(y)\dd \vrho_1(x)\\
        &=\sup_{\Psi\in{\PP_\eps}}\int_{\X\times\X} (-\Psi(x,y))\frac{\dd\vrho_1}{\dd\nu}(x)\grad_\eps[u](x,y)\dd (\nu\otimes\nu).
    \end{align}
    Summing both expressions and taking the supremum over $(\Psi^0,\Psi^1)\in {\PP_\eps}\times{\PP_\eps}$ proves the statement.
\end{proof}
We next rewrite the gradient formulation in an integration-by-parts form 
which shows that $-\grad_\eps$ is the adjoint operator of $\div_\eps$ and will be the core ingredient to prove the second part of \cref{thm:u bounded-dual representation}.
\begin{proposition}[Integration-by-parts]\label{prop:u bounded-integration by parts}
    Let $\X$ be a separable metric space with $\sigma$-finite measure~$\nu$.
    For $u\in L^\infty(\X)$ and $v\in L^1(\X\times\X)$ we have
    \begin{align}
        \int_{\X\times\X}\grad_\eps[u] \cdot v\dd(\nu\otimes\nu)=-\int_\X u\cdot\div_\eps[v]\dd\nu.
    \end{align}
\end{proposition}
\begin{proof}
    By Fubini--Tonelli we can split the integral into two iterated integrals, namely
    \begin{align}
        &\int_{\X\times\X} \grad_\eps[u]\cdot v\dd(\nu\otimes\nu)\\
        &\quad=\frac{1}{\eps}\left(\int_\X\int_\X v(x,y)\dd\nu(x) u(y) \dd\nu(y)-\int_\X\int_\X u(x) v(x,y)\dd\nu(x)\dd\nu(y)\right).
    \end{align}
    By interchanging $x$ and $y$ and applying Fubini's theorem, the second part can then be rewritten as
    \begin{align}
        \int_\X\int_\X u(x)v(x,y)\dd\nu(x)\dd\nu(y)=\int_\X\int_\X v(y,x)\dd\nu(x) u(y) \dd\nu(y)
    \end{align}
    and we obtain
    \begin{align}
        \int_{\X\times\X}\grad_\eps[u]\cdot v\dd(\nu\otimes\nu)=\int_\X\int_{\X}\frac{v(x,y)-v(y,x)}{\eps}\dd\nu(x) u(y)\dd\nu(y)=-\int_\X\div_\eps[v]\cdot u\dd\nu.
    \end{align}
    This concludes the proof.
\end{proof}
\begin{corollary}\label{cor:u bounded-nonlocal divergence}
    Under the conditions of \cref{thm:u bounded-dual representation} for all $u\in L^\infty(\X)$ and $\Psi\in{\PP_\eps}$
    we have
    \begin{align}
        \tv_\eps(u)=\sup_{\FPsi\in{\PP_\eps}\times{\PP_\eps}}-\int_{\X} u(y) \div_\eps\left[\FPsi;\Frho\right](y)\dd \nu(y).
    \end{align}
\end{corollary}
\begin{proof}
    Applying \cref{prop:u bounded-integration by parts} to \cref{prop:u bounded-nonlocal gradient} directly proves the statement.
\end{proof}
\begin{remark}[Connection to the $C(\X)$ case]
    Let $\Psi\in{\PP_\eps}$ and define a family of measures by $\Psi_x(A) \coloneq \int_A\Psi(x,y)\dd \nu(y)$ for any $x\in \X$ and $A\in\BB_\X$. Invoking Fubini--Tonelli and the properties of $\Psi$ we obtain that this family defines a random walk on $\X$ in the sense of \cref{def:random walk}. In particular, due to the support condition of $\Psi$ we have $\serifPsi=\set{\Psi_x}{x\in \X}\in {\MM_\eps}$. Furthermore, we can apply \eqref{eq:u cont-semidirect measure on rectangles} and \eqref{eq:u cont-semidirect measure} to obtain
    \begin{align}
        \div_\eps[\vrho_i\ltimes\serifPsi](A)
        &=
        \frac{(\vrho_i\ltimes \serifPsi)(A\times\X)-(\vrho_i\ltimes\serifPsi)(\X\times A)}{\eps}
        \\
        &=
        \frac{\int_A\int_{\X}\Psi(x,y)\dd \nu(y)\dd\vrho_i(x)  - \int_{\X} \int_{A}\Psi(x,y)\dd \nu(y)\dd\vrho_i(x)}{\eps}
        \\
        &=\frac{\int_A \int_{\X} \Psi(x,y)\frac{\dd\vrho_i}{\dd\nu}(x)\dd\nu(y)\dd\nu(x)-\int_A\int_{\X}\Psi(x,y)\frac{\dd\vrho_i}{\dd\nu}(x)\dd\nu(x)\dd\nu(y)}{\eps}\\
        &=
        \int_A\frac{\int_{\X}\Psi(y,x)\frac{\dd\vrho_i}{\dd\nu}(y)\dd \nu(x) - \int_{\X}\Psi(x,y)\frac{\dd\vrho_i}{\dd\nu}(x)\dd \nu(x)}{\eps}\dd \nu(y),
    \end{align}
    where we also invoked the Fubini--Tonelli theorem. 
    Hence, the measure $\div_\eps[\vrho_i\ltimes\BoldserifPsi]$ has a density (with respect to $\nu$) given by
    \begin{align}
        y\mapsto
        \int_{\X}\frac{\Psi(y,x)\frac{\dd\vrho_i}{\dd\nu}(y)-\Psi(x,y)\frac{\dd\vrho_i}{\dd\nu}(x)}{\eps}\dd \nu(x)
    \end{align}
    which coincides with the divergence $\div_\eps[\Psi\vrho_i]$ defined in \cref{def:u bounded-nonlocal divergence}.    
    In summary, we showed that for $\FPsi=(\Psi^0,\Psi^1)\in {\PP_\eps}\times{\PP_\eps}$ and the associated canonical random walk $\BoldserifPsi\in {\MM_\eps}\times{\MM_\eps}$ we have the following equality of measures:
    \begin{align}
        \div_\eps[\BoldserifPsi;\Frho]
        =
        \div_\eps[\FPsi;\Frho]\dd\nu.
    \end{align}
\end{remark}
The definitions of the nonlocal quantities are not only consistent within the two settings examined in this paper but also with their local counterparts.
More precisely, we show that in the finite dimensional setting of an open domain $\X=\Ome\subset\IRN$, the weighted Laplacian defined as $\div_\eps[\rho\grad_\eps[u]]$ is consistent with the differential operator $\div(\rho\nabla u)$, when suitably scaled.
For this consistency to be true we however have to slightly modify the definition of the gradient to be compactly supported in $\clDeps$ or alternatively to multiply it with the indicator function of $\Deps$.
Note that we did not include this into \cref{def:nonlocal_gradient_bounded} of the nonlocal gradient in the first place in order to have the definitions for continuous functions and bounded measurable functions to be the same.
Obviously, the function $(x,y)\mapsto \frac{u(y)-u(x)}{\eps}\Ie_{\clDeps}(x,y)$ is not continuous and hence such a definition would be invalid in \cref{sec:continuous}.

We remind the reader of the definition of the inner parallel set $\Ome_\eps\coloneq \set{x\in\Ome}{\mathrm{dist}(x,\IRN\setminus \Ome)>\eps}$.
\begin{proposition}\label{prop:u bounded-consistency}
    Let $\Ome\subset\IRN$ be a bounded domain. 
    For $u\in C^{2}(\overline{\Ome})$ and $\vrho\in C^1(\overline{\Ome})$ we have
    \begin{align}
        \frac{1}{C_N\eps^{N}}\div_\eps\big[\vrho\grad_\eps[u]\Ie_{\clDeps}\big]=\div\big(\vrho\nabla u\big)+o(1)\quad\text{on }\Ome_\eps\text{ as $\eps\to 0$},
    \end{align}
    where $C_N$ is a constant only depending on the dimension $N$ and the convergence is uniform on compactly contained subsets of $\Omega$.
    Note that, analogous to before, $(\rho\grad_\eps[u]\Ie_{\clDeps})(x,y) := \rho(x)\grad_\eps[u](x,y)\Ie_{\clDeps}(x,y)$ for $x,y\in\Ome$.
\end{proposition}
\begin{proof}
    For $y\in \Ome_\eps$ we have 
    \begin{align}
        \div_\eps\big[\vrho\grad_\eps[u]\Ie_{\clDeps}\big](y)=\frac{1}{\eps^2}\int_{\ball{\eps}{y}} \left(u(x)-u(y)\right)\left(\vrho(x)+\vrho(y)\right)\dd x
    \end{align}
    and an application of Taylor's expansion to both factors yields
    \begin{align}
        \div_\eps\big[\vrho&\grad_\eps[u]\Ie_{\clDeps}\big](y)\\
        =&\frac{2}{\eps^2}\int_{\ball{\eps}{y}}\vrho(y)\dotproduct{\nabla u(y)}{x-y}\dd x+\frac{1}{\eps^2}\int_{\ball{\eps}{y}}\vrho(y)\dotproduct{x-y}{\nabla^2 u(y) (x-y)}\dd x\\
        &+\frac{1}{\eps^2}\int_{\ball{\eps}{y}} \dotproduct{\nabla u(y)}{x-y}
        \dotproduct{\nabla\rho(y)}{x-y} \dd x\\&+ \frac{1}{\eps^2}\int_{\ball{\eps}{y}}\frac{1}{2}\dotproduct{x-y}{\nabla^2 u(y) (x-y)}\dotproduct{\nabla\rho(y)}{x-y}\dd x +o(\eps^N)
        \\\eqcolon&I_1+I_2+I_3+I_4+o(\eps^N).
    \end{align}
    A direct calculation involving a change of variables and spherical integration shows that
    \begin{align}
        \eps^2 I_2&=\sum_{i,j=1}^N\p_{ij}u(y)\vrho(y)\int_{\ball{\eps}{y}} (x_i-y_i)(x_j-y_j)\dd x=\sum_{i,j=1}^N\p_{ij}u(y)\vrho(y)\eps^{N+2}\int_{\ball{1}{0}} z_iz_j\dd x\\&=\eps^{N+2}\sum_{i,j=1}^N\p_{ij}^2 u(y) \delta_{ij} C_N \vrho(y)=\eps^{N+2} C_N  \Delta u(y)\vrho(y)
    \end{align}
    and
    \begin{align}
        \eps^2I_3&=\sum_{i,j=1}^N\p_i u(y)\p_j \vrho(y)\int_{\ball{\eps}{y}}(x_i-y_i)(x_j-y_j)\dd x=\sum_{i,j=1}^N\p_i u(y)\p_j \vrho(y)\eps^{N+2}\int_{\ball{1}{0}}z_i z_j\dd z\\&=\eps^{N+2} C_N\sum_{i,j=1}^N\delta_{ij} \p_i u(y) \p_j \vrho(y)=\eps^{N+2}C_N\dotproduct{\nabla u(y)}{\nabla \vrho(y)}
    \end{align}
    where $C_N\coloneq \frac{1}{N+2}\frac{2 \pi^{\frac{N}{2}}}{\Gamma(\frac{N}{2})N}$ is a multiple of the volume of the unit ball in $\R^N$. Furthermore, by odd symmetry $I_1=I_4=0$ and therefore
    \begin{align}
		  \div_\eps\big[\vrho\grad_\eps[u]\Ie_{\clDeps}\big](y)&=\eps^N C_N\left(\Delta u(y)\vrho(y)+\nabla u(y)\nabla \vrho(y)\right)+o(\eps^{N})\\&=\eps^N C_N \div(\vrho\nabla u)(y)+o(\eps^{N}).
    \end{align}
    Dividing by $C_N\eps^N$ then yields the result, noting that all error terms can be chosen to be independent of $y$ for $y$ in a compactly contained subset of $\Omega$ since $u\in C^2(\overline{\Omega})$ and $\rho\in C^1(\overline{\Omega})$.
\end{proof}
In particular, this yields a distributional convergence.
\begin{corollary}[Distributional convergence]
    Let $\Ome\subset\IRN$ be a bounded domain, $\vrho\in C^1(\overline{\Ome})$ and $u\in L^\infty(\Ome)$. Then, for every $\phi\in \Cci(\Ome)$ we have
    \begin{align}
        \lim_{\eps\to\infty} \frac{1}{C_N\eps^N}\int_\Ome \div_\eps\big[\vrho\grad_\eps[u]\Ie_{\clDeps}\big]\phi\dd x =\int_\Ome u \div(\vrho\nabla \phi)\dd x
    \end{align}
    where $C_N$ is the constant originating from \cref{prop:u bounded-consistency}.
    This means that for $u\in L^\infty(\Ome)$ we have
    \begin{align}
        \lim_{\eps\to\infty} \frac{1}{C_N\eps^N}\div_\eps\big[\vrho\grad_\eps[u]\Ie_{\clDeps}\big]
        =
        \div(\rho\nabla u)
    \end{align}
    in the sense of distributions, where the right hand side is to be understood as a distribution.
\end{corollary}
\begin{proof}
    Applying \cref{prop:u bounded-integration by parts} two times, then for $\eps>0$ we obtain
    \begin{align}
        \int_\Ome\div_\eps\big[\vrho\grad_\eps[u]\Ie_{\clDeps}\big]\phi\dd x
        &
        =-\int_{\Ome\times\Ome}\vrho\grad_\eps[u]\Ie_{\clDeps} \grad_\eps[\phi]\dd(x,y)
        \\
        &=\int_\Ome u \div_\eps\big[\vrho\grad_\eps[\phi]\Ie_{\clDeps}\big]\dd x.
    \end{align}
    Here, the first application is justified because $u\in L^\infty(\Ome)$ and thus $\grad_\eps[u]\in L^\infty(\Ome\times\Ome)$ which together with $\rho\in L^\infty(\Ome)$ trivially implies that $\rho\grad_\eps[u]\Ie_{\clDeps}\in L^1(\Ome\times\Ome)$.
    Furthermore, trivially $\phi\in L^\infty$.
    For the second application of \cref{prop:u bounded-integration by parts} note that $\grad_\eps[\phi]$ is controlled by the derivative of $\phi$ due to the mean value theorem, which itself is bounded. It is now left to show that the distributional limit of the last integral converges as stated. First, since $\phi\in \Cci(\Ome)$ we can extend it continuously by zero such that $\phi\in C^2(\overline{\Ome})$. Because $\phi$ has compact support we can choose $\eps$ small enough such that $\phi$ vanishes on $\Omega\setminus\Omega_{2\eps}$ implying that $\div_\eps\big[\vrho\grad_\eps[\phi]\Ie_{\clDeps}\big]=\div(\vrho\nabla\phi)=0$ on $\Ome\setminus\Ome_\eps$. Hence, by \cref{prop:u bounded-consistency} we have
    \begin{align}
        \mathsf{S}_\eps\coloneq \sup_{\Ome}\abs{\frac{1}{C_N\eps^N}\div\big[\vrho\grad_\eps[\phi]\Ie_{\clDeps}\big]-\div(\vrho\nabla \phi)}\xrightarrow[\eps\to 0]{}0.
    \end{align}
    Consequently, for $u\in L^\infty(\Ome)$ we obtain
    \begin{align}
        \abs{\int_\Ome \Big(\frac{1}{C_N\eps^N}\div\big[\vrho\grad_\eps[\phi]\Ie_{\clDeps}\big]-\div(\vrho\nabla\phi)\Big) u\dd x}\leq \abs{\Ome}\cdot\norm{u}_{L^\infty(\Ome)} \cdot\mathsf{S}_\eps \xrightarrow[\eps\to 0]{}0
    \end{align}
    because $\Ome$ is assumed to be bounded.
\end{proof}

\subsection{Limit Characterization of the Subdifferential}\label{sec:subdiff_bounded}

    Using the same arguments as in the proof of \cref{thm:u cont-subdiff characterization}, we have that $\mu^*\in\partial\tv_\eps(u)$ for $u\in L^\infty(\X)$ if and only if 
    \begin{align}
        \mu^*\in \overline{\operatorname{conv}}\set{-\div_\eps[\FPsi;\Frho]}{\FPsi\in {\PP_\eps}\times{\PP_\eps}}
        \subset L^\infty(\X)^*\qquad\text{and}\qquad \tv_\eps(u)=\int_{\X} u \dd \mu^*,
    \end{align}
    where the closure is taken with respect to the weak-* topology in $L^\infty(\X)^*$, the topological dual space of $L^\infty(\X)$. 
    In \cref{sec:continuous} the set of divergences of test functions was already convex and weak-* closed yielding a straightforward pointwise characterization of the subdifferential. 
    In this section, the situation is different due to a lack of closedness of the set $P\coloneq\set{-\div_\eps[\FPsi;\Frho]}{\FPsi\in {\PP_\eps}\times{\PP_\eps}}$.
    A direct calculation shows that $P$ is convex and $\norm{p}_{L^1(\X)}\leq\frac{2}{\eps}$ for every $p\in P$ implying that $P\subset L^1(\X)$.
    As a consequence of the Banach--Alaoglu theorem, $\mathcal{K}\coloneq \overline{\mathrm{conv}}(P)\subset L^\infty(\X)^*$ is \mbox{weak-*} compact, albeit \emph{not necessarily sequentially \mbox{weak-*} compact}.
    This implies that a weak-* continuous function attains its maximum over $\mathcal{K}$ and due to \cref{thm:u bounded-dual representation} we have
    \begin{align}
        \tv_\eps(u)=\sup_{p\in P} \int_\X u p\dd\nu= \max_{p\in\mathcal{K}}\left\langle p,u\right\rangle_{(L^\infty)^*\times L^\infty}
    \end{align}
    for any $u\in L^\infty(\X)$, where on the right hand side we have the dual pairing between $L^\infty(\X)^*$ and $L^\infty(\X)$.
    So, for an abstract characterization, the subdifferential consists of all maximizers in $\mathcal{K}$ of this pairing, in particular $\p\tv_\eps(u)\neq \emptyset$ for all $u\in L^\infty(\X)$ since $\emptyset\neq P \subset \mathcal{K}$. 
    However, for an explicit characterization one would need to further investigate the weak*-closure of $P$.
    The dual space $L^\infty(\X)^*$ is known to be the space of all finitely additive finite signed measures on $\BB_{\X}$ that are absolutely continuous with respect to $\nu$, denoted by $ba(\X,\BB_{\X},\nu)$; see \cite[Theorem IV.8.16]{dunford1988}. 
    Note that, since $L^\infty(\X)$ is not separable  unless $\X$ is finite, sequential and topological weak-* compactness do not coincide.
    The lack of sequential compactness requires one to work with nets rather than sequences, since a point belongs to the closure of a set in a topological space if and only if it is the limit of a net in the set, see \cite[Theorem 2.14]{aliprantis2006}.
    
    We shortly recall the definition and some properties of a net, following \cite[Chapter 2.4]{aliprantis2006} where a complete introduction is given.
    For a topological space $(X,\tau)$ a \emph{net} is a function $x\colon \mathfrak{D}\to X$, where $\mathfrak{D}$ is a directed set.
    This refers to a set with a reflexive transitive binary relation $\preceq$, called direction, which enjoys the property that each pair has an upper bound, i.e., for all $x,y\in X$ there exists $z\in X$ with $x\preceq z$ and $y\preceq z$. 
    In particular, sequences are nets where $\mathfrak{D}=\IN$. 
    A net $(x_\alpha)_{\alpha\in D}\subset X$ converges to some point $x\in X$ if for every open neighborhood $U$ of $x$ there is an index $\alpha_0$, dependent on $U$, such that $x_\alpha\in U$ for all $\alpha\geq \alpha_0$.
    Note that a net in the dual space $X^*$ converges in the weak-* topology if and only if its dual pairing with any element in $X$ is converging as a net in $\IR$; see the paragraph below \cite[Definition 5.90]{aliprantis2006}.
\begin{theorem}[Subdifferential characterization]\label{thm:u bounded-subdiff char}
    Let $\X$ be a separable metric space, $\mu\in\P(\X\times\{0,1\})$ a probability measure, and assume that the respective conditional distributions $\vrho_i\coloneq\mu(\filll\times\{i\})$ for $i\in\{0,1\}$ are absolutely continuous with respect to a $\sigma$-finite reference measure $\nu$. Furthermore, let $u\in L^\infty(\X)$. Then, $\mu^*\in \p\tv_\eps(u)$ if and only if
    \begin{align}
            \begin{dcases}
                &\mu^*\in ba(\X,\BB_{\X},\nu), \quad\tv_\eps(u)=\int_\Ome u \dd \mu^*,\quad\text{and}
                \\
                & \text{there exists a net } ({\FPsi}_\alpha)_{\alpha\in \mathfrak{D}}\text{ in }{\PP_\eps}\times{\PP_\eps} \text{ such that}
                \\ 
                &-\int_\Ome v\div_\eps[\FPsi_\alpha;\Frho]\dd \nu(x)\to \int_\Ome v\dd \mu^*\text{ for all }v\in L^\infty(\Ome).
            \end{dcases}
    \end{align}
\end{theorem}
We illustrate the necessity of subgradients outside of $L^1(\X)$ in the following example.
\begin{example}[Non-existence of $L^1$-subgradients]\label{ex:nonexistence of subgrad}
    Consider $\X=(-1,1)$, $\eps=1$, $\rho_0=\nu=\mathcal{L}^1\restr(-1,1)$, and $\rho_1=0$.
    Then it holds that
    \begin{align}
        \tv_\eps(u) 
        = \int_{-1}^1
        \esssup{u}{(x-1,x+1)}\d x
        =\esssup{u}{(-1,1)}.
    \end{align}
    Let us take a function $u\in L^1((-1,1))\cap C((-1,1))$ which attains its global maximum only at $x=0$.
    For simplicity, we will work with $u(x)=1-|x|$.
    If $p\in L^1((-1,1))$ was a subgradient of $\tv_\eps(u)$, then it would hold
    \begin{align}
        \int_{-1}^1 p u \dx =  u(0)
        \qquad
        \int_{-1}^1 p v \dx \leq  \esssup{v}{(-1,1)}\;\forall v \in L^1((-1,1)).
    \end{align}
    By using $v\equiv 1$ these two properties imply that $\int_{-1}^1 p \d x = 1$.
    Furthermore, we claim $p\geq 0$.
    To see this let $x_0\in(-1,1)$ and $0<\delta<1$ so small that $[x_0-\delta,x_0+\delta]\subset(-1,1)$.
    Using $v_\delta\coloneq -\Ie_{(x_0-\delta,x_0+\delta)}$ as a test function we obtain
    \begin{align}
        -\int_{x_0-\delta}^{x_0+\delta}p(x)\d x \leq \esssup{v_\delta}{(-1,1)}=0
    \end{align}
    and hence
    \begin{align}
        \int_{x_0-\delta}^{x_0+\delta}p\d x\geq 0.
    \end{align}
    Since this holds for any $x_0\in(-1,1)$ and any $\delta>0$ sufficiently small, we find $p\geq 0$ Lebesgue-almost everywhere on $(-1,1)$.
    Finally, using that $u(0)=1$ we have that
    \begin{align}
        1 = \int_{-1}^1p(x)u(x)\d x = \int_{-1}^1p(x)\d x
    \end{align}
    and therefore $\int_{-1}^1 p(x)(1-u(x))\d x = \int_{-1}^1p(x)|x|\d x=0$.
    Since $p\geq 0$, this shows that $p=0$ Lebesgue-almost everywhere on $(-1,1)$, a contradiction.
    
    In fact, a valid subgradient is given by the bounded finitely additive measure
    \begin{align}
        p(A) := \lim_{\delta\to 0} \frac{1}{2\delta}\mathcal{L}^1\left(A\cap(-\delta,\delta)\right)
    \end{align}
    for which we have
    \begin{align}
        \int_{-1}^1 v \d p = \lim_{\delta\to\infty}
        \frac{1}{2\delta}
        \int_{-\delta}^\delta v\d x \leq \esssup{v}{(-1,1)} = \tv_\eps(v)
    \end{align}
    and by the continuity of $u$ (at zero) we have
    \begin{align}
        \int_{-1}^1 u \d p = \lim_{\delta\to\infty}
        \frac{1}{2\delta}
        \int_{-\delta}^\delta u\d x =u(0)=\tv_\eps(u).
    \end{align}
    This shows that indeed that $\partial\tv_\eps(u)\neq\emptyset$ but $\partial\tv_\eps(u)\cap L^1((-1,1))=\emptyset$.
\end{example}
\begin{remark}[Difficulties and Conjecture]
    Calculating the precise form of subgradients $\mu^*$ is delicate since, in contrast to ${\MM_\eps}\times{\MM_\eps}$ in \cref{thm:u cont-subdiff characterization}, the set ${\PP_\eps}\times{\PP_\eps}$ in \cref{thm:u bounded-subdiff char} is not closed, see \cref{ex:nonexistence of subgrad}. 
    Furthermore, as explained above, in $L^\infty(\X)^*$ sequential weak-* closedness does not coincide with weak-* closedness which is why we needed to switch to the more general notion of nets instead of sequences which provide the equivalence of closedness and sequential closedness even in general topological spaces as mentioned above. 
    We conjecture that subgradients are given by nonlocal divergences of random walks that are $ba$ measures and whose nonlocal divergence is still a $ba$ measure, comparable to the result in \cref{sec:continuous}.
    Proving this, however, is beyond the scope of this work.
\end{remark}

\section{Conclusions}\label{sec:conclusions}

In this paper, we derived dualization formulas for the nonlocal adversarial total variation functional in two different settings and characterized their respective subdifferentials. Using measure-theoretic tools, we provided a new perspective on adversarial total variation that opens new doors for the analytical and algorithmic treatment of the adversarial training problem. 
In particular, we obtained an integral characterization of its subgradients for the separable base space $C_0(\X)$ as nonlocal divergences of random walks. 
On the larger space $L^\infty(\X)$, we were still able to characterize the subdifferential in a limit sense. Furthermore, we showed that the nonlocal divergences appearing in both settings are consistent with each other and (as $\eps\to 0$) with local differential operators.

We highlight several directions for future research based on this work:

First, it would be interesting to investigate to what extent our subdifferential characterization using $L^\infty(\X)$ as base space can be further refined. 
This could be pursued by explicitly analyzing the weak-* closure and computing the limit points of the nets of test functions in order to obtain an integral characterization. 

Second, one could investigate nonlocal versions of the Anzellotti pairing \cite{anzellotti1983pairings} to derive a pointwise characterization of the subgradients similar to the one in \cite{bredies2016pointwise}.
As mentioned in the introduction, for the classical total variation Anzellotti pairings are used to define both a normal trace operator and a full trace operator, allowing the integral condition to be replaced by a pointwise trace condition, see \cite{bredies2016pointwise}. 
In the spirit of transferring successful ideas originally developed for classical total variation to the nonlocal setting, this appears to be a promising direction.

Lastly, an important part of the motivation for this work was to open the doors for new algorithmic approaches to solve the adversarial training problem more efficiently. 
As explained in the introduction, primal--dual algorithms have been successfully applied to problems with similar structure to adversarial training, however, involving the local total variation. Therefore, a natural next step is to exploit the results of this paper to develop a primal--dual algorithm for adversarial training.

\section*{Acknowledgments}

The authors acknowledge funding by the Deutsche Forschungsgemeinschaft (DFG, German Research Foundation) – project number 544579844 (GeoMAR) within DFG-SPP 2298 ``Theoretical Foundations of Deep Learning''.
Furthermore, they are grateful to the Casa Matemática Oaxaca (CMO) for hosting them during the workshop \emph{Mathematical Analysis of Adversarial Machine Learning} (25w5469) in August 2025.
LB would like to thank Antonin Chambolle for fruitful discussions around the subdifferential of the adversarial total variation which happened during the Oberwolfach workshop 2349
\emph{Variational Methods for Evolution}.
The authors would like to thank an anonymous reviewer for suggesting them alternative proofs of \cref{lem:u cont-weak measurable correspondence,thm:dual representation for bounded functions} which avoid the relatively restrictive assumptions we used in the first draft of the article (\href{https://arxiv.org/abs/2604.18540v1}{\tt arXiv: 2604.18540v1}).
The localization argument for defining $g$ in the proof of \cref{thm:dual representation for bounded functions} was obtained by Claude Opus 5.5 and verified of the authors.

\printbibliography

\end{document}